\documentclass[11pt]{article}

\usepackage{amsmath}
\usepackage{amsthm}
\usepackage{amssymb}
\usepackage{latexsym}
\usepackage{mathabx}
\usepackage{pstricks}
\usepackage{graphicx, pst-plot, pst-node, pst-text, pst-tree}
\usepackage{titlefoot}
\usepackage{enumerate}
\usepackage{titlesec}
\usepackage{units}
\usepackage[small,it]{caption}

\usepackage{color}
\definecolor{lightgray}{rgb}{0.9, 0.9, 0.9}
\definecolor{darkgray}{rgb}{0.7, 0.7, 0.7}
\definecolor{darkblue}{rgb}{0, 0, .4}

\usepackage[bookmarks,breaklinks]{hyperref}
\hypersetup{
	colorlinks=true,
	linkcolor=darkblue,
	anchorcolor=darkblue,
	citecolor=darkblue,
	pagecolor=darkblue,
	urlcolor=darkblue,
	pdfpagemode=UseThumbs,
	pdftitle={Small permutation classes},
	pdfsubject={Combinatorics},
	pdfauthor={Vincent Vatter},
	pdfkeywords={todo}
}

\newtheorem{theorem}{Theorem}[section]
\newtheorem{proposition}[theorem]{Proposition}
\newtheorem{lemma}[theorem]{Lemma}

\newtheorem{conjecture}[theorem]{Conjecture}

\newtheorem{problem}[theorem]{Problem}

\newcounter{todocounter}

\newcommand{\minisec}[1]{\bigskip\noindent{\bf #1.}}

\makeatletter
\newfont{\footsc}{cmcsc10 at 8truept}
\newfont{\footbf}{cmbx10 at 8truept}
\newfont{\footrm}{cmr10 at 10truept}
\renewenvironment{abstract}%
		{
		  \begin{list}{}%
		     {\setlength{\rightmargin}{1in}%
		      \setlength{\leftmargin}{1in}}%
		   \item[]\ignorespaces\begin{small}}%
		 {\end{small}\unskip\end{list}}

\datefoot{\today}
\amssubj{05A05, 05A15, 06A06}

\newpagestyle{main}[\small]{
	\headrule
	\sethead[\usepage][][]
	{\sc Small Permutation Classes}{}{\usepage}}

\title{\sc{Small Permutation Classes}}
\author{\sc{Vincent Vatter}\thanks{The first draft of this paper was completed while the author was a research fellow at the University of St Andrews, supported by EPSRC grant EP/C523229/1.  Later revisions were completed while the author was a John Wesley Young Research Instructor at Dartmouth College.}\\
\small Department of Mathematics\\[-3pt]
\small University of Florida\\[-3pt]
\small Gainesville, Florida 32611 USA\\[-10pt]}

\titleformat{\section}
	{\large\sc}
	{\thesection.}{1em}{}	

\titleformat{\subsection}
	{\normalsize\sc}
	{\thesubsection.}{1em}{}	

\date{}

\begin{document}
\maketitle

\pagestyle{main}

\newcommand{\Av}{\operatorname{Av}}
\newcommand{\Age}{\operatorname{Age}}
\newcommand{\A}{\mathcal{A}}
\newcommand{\C}{\mathcal{C}}
\newcommand{\D}{\mathcal{D}}
\newcommand{\E}{\mathcal{E}}
\newcommand{\HH}{\mathcal{H}}
\newcommand{\I}{\mathcal{I}}
\newcommand{\J}{\mathcal{J}}
\newcommand{\K}{\mathcal{K}}
\renewcommand{\L}{\mathcal{L}}
\newcommand{\M}{\mathcal{M}}
\newcommand{\N}{\mathcal{N}}
\renewcommand{\P}{\mathcal{P}}
\newcommand{\R}{\mathcal{R}}
\renewcommand{\S}{\mathcal{S}}
\renewcommand{\O}{\mathcal{O}}
\renewcommand{\S}{\mathcal{S}}
\newcommand{\gr}{\mathrm{gr}}
\newcommand{\lgr}{\underline{\gr}}
\newcommand{\ugr}{\overline{\gr}}
\newcommand{\Grid}{\operatorname{Grid}}
\newcommand{\zpm}{0/\mathord{\pm} 1}
\newcommand{\proj}{\operatorname{proj}}
\newcommand{\height}{\operatorname{ht}}
\newcommand{\OEISlink}[1]{\href{http://www.research.att.com/projects/OEIS?Anum=#1}{#1}}
\newcommand{\OEISref}{\href{http://www.research.att.com/\~njas/sequences/}{OEIS}~\cite{sloane:the-on-line-enc:}}
\newcommand{\OEIS}[1]{(Sequence \OEISlink{#1} in the \OEISref.)}

\begin{abstract}
We establish a phase transition for permutation classes (downsets of permutations under the permutation containment order): there is an algebraic number $\kappa$, approximately $2.20557$, for which there are only countably many permutation classes of growth rate (a.k.a. Stanley-Wilf limit) less than $\kappa$ but uncountably many permutation classes of growth rate $\kappa$, answering a question of Klazar.  We go on to completely characterize the possible sub-$\kappa$ growth rates of permutation classes, answering a question of Kaiser and Klazar.  Central to our proofs are the concepts of generalized grid classes (introduced herein), partial well-order, the substitution decomposition, and atomicity (a.k.a. the joint embedding property).
\end{abstract}

\section{Introduction}\label{small-intro}

For any collection (also known as a {\it property\/}), $\P$, of finite combinatorial structures, the function which maps $n$ to the number of structures in $\P$ with ground set $[n]=\{1,2,\dots,n\}$ is known as the {\it speed\/} of $\P$.  A property $\P$ is further said to be {\it hereditary\/} if it is closed under isomorphism and substructures.  The study of the possible speeds of hereditary properties of combinatorial structures essentially dates back to Scheinerman and Zito~\cite{scheinerman:on-the-size-of-:}, who established the first significant characterization theorem: speeds of hereditary properties of labeled graphs must be bounded, polynomial, exponential, factorial, or superfactorial.  Since this seminal result, speeds of various combinatorial structures have been studied; see Bollob\'as~\cite{bollobas:hereditary-and-} and Klazar~\cite{klazar:overview-of-som} for surveys.

Our interest lies with hereditary properties of permutations, which we call \emph{permutation classes}.  The permutation $\pi$ of $[n]$ \footnote{Here $[n]=\{1,2,\ldots,n\}$ and, more generally, for $a,b \in\mathbb{N}$ with $a<b$, the interval $\{a,a+1,\ldots,b\}$ is denoted by $[a,b]$, the interval $\{a+1,a+2,\ldots,b\}$ by
$(a,b]$, and so on.} contains the permutation $\sigma$ of $[k]$ (written $\sigma\le\pi$) if $\pi$ has a subsequence of length $k$ which is order isomorphic to $\sigma$, and such a subsequence is called an {\it occurrence\/}, or {\it copy\/}, of $\sigma$.  For example, $\pi=391867452$ (written in list, or one-line notation) contains $\sigma=51342$, as can be seen by considering the subsequence $91672$ ($=\pi(2),\pi(3),\pi(5),\pi(6),\pi(9)$).  A {\it permutation class\/} is thus a downset of permutations under this order: if $\C$ is a permutation class, $\pi\in\C$, and $\sigma\le\pi$, then $\sigma\in\C$.  For any set $X$ of permutations, we define its {\it closure\/} to be the permutation class $\{\sigma : \mbox{$\sigma\le\pi$ for some $\pi\in X$}\}$.  For any permutation class $\C$ there is a unique (and possibly infinite) antichain $B$ such that $\C=\Av(B)=\{\pi: \pi \not \geq\beta\mbox{ for all } \beta \in B\}$. This antichain $B$ is called the {\it basis} of $\C$.

We denote by $\C_n$ ($n \in \mathbb{N}$) the set of permutations in $\C$ of length $n$, so the speed of $\C$ is the function $n\mapsto |\C_n|$.  We further refer to $\sum |\C_n|x^n$ as the {\it generating function for $\C$\/}.  Whether this generating function counts the empty permutation of length $0$ is a matter a taste; we elect to count it except when noted.

The Marcus-Tardos Theorem~\cite{marcus:excluded-permut:} (formerly the Stanley-Wilf Conjecture) states that all permutation classes other than the class of all permutations have at most exponential speed, i.e., for every class $\C$ with a nonempty basis, there is a constant $K$ so that $\C$ contains at most $K^n$ permutations of length $n$ for all $n$.  Thus every nondegenerate permutation class $\C$ has finite {\it upper\/} and {\it lower growth rates\/} defined, respectively, by
\begin{eqnarray*}
\ugr(\C)&=&\limsup_{n\rightarrow\infty}\sqrt[n]{|\C_n|}\mbox{ and}\\
\lgr(\C)&=&\liminf_{n\rightarrow\infty}\sqrt[n]{|\C_n|}.
\end{eqnarray*}
It is conjectured that every permutation class has a {\it growth rate\/}, and when we are dealing with a class for which $\ugr(\C)=\lgr(\C)$, we denote this quantity by $\gr(\C)$.

Herein we are concerned with the set of (upper, lower) growth rates of permutation classes.  From this viewpoint, the Erd\H os-Szekeres Theorem below characterizes the permutation classes of growth rate $0$: if $\C$ contains arbitrarily long monotone permutations then $\lgr(\C)\ge 1$ while otherwise $\C$ is finite and so $\gr(\C)$ exists and equals $0$.

\newtheorem*{es}{The Erd\H os-Szekeres Theorem~\cite{erdos:a-combinatorial:}}
\begin{es}
Every permutation $\pi$ of length $(m-1)^2+1$ contains a monotone permutation of length at least $m$.
\end{es}

Moving beyond the growth rate $0$ classes, Kaiser and Klazar~\cite{kaiser:on-growth-rates:} proved that the only growth rates of permutation classes (they are in this case proper, unadjectivated, growth rates) less than $2$ are positive roots of $1-2x^k+x^{k+1}$ for some $k\ge 0$, i.e., that the speeds of such classes are logarithmically asymptotic to the $k$-Fibonacci numbers for some $k$.  Later, Klazar~\cite{klazar:on-the-least-ex:} showed that there are only countably many permutation classes with these growth rates.  Two natural questions are then
\begin{enumerate}
\item Kaiser and Klazar~\cite{kaiser:on-growth-rates:}: What is the next (lower, upper, proper) growth rate after $2$?
\item Klazar~\cite{klazar:on-the-least-ex:}: What is the smallest (lower, upper, proper) growth rate for which there are uncountably many permutation classes?
\end{enumerate}
We provide answers to both of these:  Answer 1 (see Theorem~\ref{small-growth-rates}) is
$$
\nu=\mbox{the unique positive root of $1+2x+x^2+x^3-x^4$}\approx 2.06599
$$
while Answer 2 (see Theorem~\ref{theorem-kappa}) is
$$
\kappa=\mbox{the unique positive root of $1+2x^2-x^3$}\approx 2.20557.
$$
Furthermore, as established in Theorem~\ref{small-growth-rates}, $\kappa$ is the least accumulation point of accumulation points of growth rates of permutation classes.  We call permutations classes with growth rates less than $\kappa$ {\it small\/}.

Central to our proofs is the concept of generalized grid classes, introduced in Section~\ref{sec-generalized-grids}.  In Section~\ref{sec-gridding-characterization} we present a characterization of these classes and of grid irreducible classes.  We then show in Section~\ref{sec-subst-gridding} that every small permutation class is $\D$-griddable for  nice class $\D$ and in Section~\ref{sec-gridding-condition}, that small permutation classes lie in particularly tractable grid classes.  These results allow us in Section~\ref{sec-small-growth-rates} to provide Answer 2, i.e., that there are only countably many permutation classes with growth rate less than $\kappa$, and as a consequence, that each of these classes is partially well-ordered.  This partial well-order condition turns out to be essential for the characterization of growth rates below $\kappa$ (these all happen to be proper growth rates), also in Section~\ref{sec-small-growth-rates}.

Quite often we need to show that small permutation classes cannot contain certain types of structures, or in other words, that if a class were to contain those structures then it would be large.  These computations, which are in the cases we need mostly routine but always tedious, have been relegated to the Appendix.

The remainder of the introduction consists of basic structural notions used in the proof.  For completeness, short proofs of the lesser known results have been liberally included.  Also, for concreteness, the results are specialized to the case of permutations, although in many cases they hold much more generally.

\minisec{Symmetries}
It is frequently helpful to note that permutation containment is preserved by the symmetries of the square, i.e., the dihedral group on $8$ elements.  This group is generated by the three symmetries inverse, reverse, and complement, which are defined, respectively, on a permutation $\pi$ of length $n$ by
\begin{eqnarray*}
\pi^{-1}(\pi(i))&=&i,\\
\pi^{\textrm{r}}(i)&=&\pi(n+1-i),\\
\pi^{\textrm{c}}(i)&=&n+1-\pi(i),
\end{eqnarray*}
for all $i$.

\minisec{Inflations, simple permutations, alternations, and substitution completions}
An {\it interval\/} in the permutation $\pi$ is a set of contiguous indices $I=[a,b]$ such that the set of values $\pi(I)=\{\pi(i) : i\in I\}$ is also contiguous.  Every permutation $\pi$ of $[n]$ has intervals of length $0$, $1$, and $n$; $\pi$ is said to be {\it simple\/} if it has no other intervals.  For an extensive study of simple permutations we refer the reader to Brignall's survey~\cite{brignall:a-survey-of-sim:}.

Going in the other direction, given a permutation $\sigma$ of length $m$ and nonempty permutations $\alpha_1,\dots,\alpha_m$, the {\it inflation\/} of $\sigma$ by $\alpha_1,\dots,\alpha_m$ --- denoted $\sigma[\alpha_1,\dots,\alpha_m]$ --- is the permutation obtained by replacing each entry $\sigma(i)$ by an interval that is order isomorphic to $\alpha_i$.  For example, $2413[1,132,321,12]=4\ 798\ 321\ 56$ (see Figure~\ref{fig-479832156}).  Simple permutations cannot be deflated.  Conversely, we have the following.

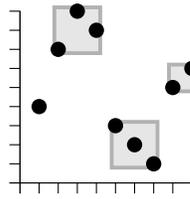
\begin{figure}
\begin{center}
\psset{xunit=0.01in, yunit=0.01in}
\psset{linewidth=0.005in}
\begin{pspicture}(0,0)(93,93)
\psaxes[dy=10, Dy=1, dx=10, Dx=1, tickstyle=bottom, showorigin=false, labels=none](0,0)(90,90)
\psframe[linecolor=darkgray,fillstyle=solid,fillcolor=lightgray,linewidth=0.02in](7,37)(13,43)
\psframe[linecolor=darkgray,fillstyle=solid,fillcolor=lightgray,linewidth=0.02in](17,67)(43,93)
\psframe[linecolor=darkgray,fillstyle=solid,fillcolor=lightgray,linewidth=0.02in](47,7)(73,33)
\psframe[linecolor=darkgray,fillstyle=solid,fillcolor=lightgray,linewidth=0.02in](77,47)(93,63)
\pscircle*(10,40){0.04in}
\pscircle*(20,70){0.04in}
\pscircle*(30,90){0.04in}
\pscircle*(40,80){0.04in}
\pscircle*(50,30){0.04in}
\pscircle*(60,20){0.04in}
\pscircle*(70,10){0.04in}
\pscircle*(80,50){0.04in}
\pscircle*(90,60){0.04in}
\end{pspicture}
\end{center}
\caption{The plot of $479832156$, an inflation of $2413$.}\label{fig-479832156}
\end{figure}

\begin{proposition}[Albert and Atkinson~\cite{albert:simple-permutat:}]\label{simple-decomp-1}
Every permutation except $1$ is the inflation of a unique simple permutation of length at least $2$.  Furthermore, if $\pi$ can be expressed as $\sigma[\alpha_1,\dots,\alpha_m]$ where $\sigma$ is a nonmonotone simple permutation, then each interval $\alpha_i$ is unique.
\end{proposition}

A class $\C$ of permutations is {\it substitution complete\/} if $\sigma[\alpha_1,\dots,\alpha_m]\in\C$ for all $\sigma\in\C_m$ and $\alpha_1,\dots,\alpha_m\in\C$.  The {\it substitution completion\/} of a set $X$ of permutations, denoted $\S(X)$, is defined as the smallest substitution complete class containing $X$.  Note that if $\C$ is a permutation class then $\C$ and $\S(\C)$ contain the same simple permutations.

\begin{figure}
\begin{center}
\begin{tabular}{ccccc}
\psset{xunit=0.01in, yunit=0.01in}
\psset{linewidth=0.005in}
\begin{pspicture}(0,0)(100,100)
\psline[linecolor=darkgray,linestyle=solid,linewidth=0.02in](55,0)(55,105)
\psaxes[dy=10,Dy=1,dx=10,Dx=1,tickstyle=bottom,showorigin=false,labels=none](0,0)(100,100)
\pscircle*(10,70){0.04in}
\pscircle*(20,50){0.04in}
\pscircle*(30,30){0.04in}
\pscircle*(40,90){0.04in}
\pscircle*(50,10){0.04in}
\pscircle*(60,60){0.04in}
\pscircle*(70,100){0.04in}
\pscircle*(80,40){0.04in}
\pscircle*(90,80){0.04in}
\pscircle*(100,20){0.04in}
\end{pspicture}
&\rule{10pt}{0pt}&
\psset{xunit=0.01in, yunit=0.01in}
\psset{linewidth=0.005in}
\begin{pspicture}(0,0)(100,100)
\psline[linecolor=darkgray,linestyle=solid,linewidth=0.02in](55,0)(55,105)
\psaxes[dy=10,Dy=1,dx=10,Dx=1,tickstyle=bottom,showorigin=false,labels=none](0,0)(100,100)
\pscircle*(10,10){0.04in}
\pscircle*(20,30){0.04in}
\pscircle*(30,50){0.04in}
\pscircle*(40,70){0.04in}
\pscircle*(50,90){0.04in}
\pscircle*(60,20){0.04in}
\pscircle*(70,40){0.04in}
\pscircle*(80,60){0.04in}
\pscircle*(90,80){0.04in}
\pscircle*(100,100){0.04in}
\end{pspicture}
&\rule{10pt}{0pt}&
\psset{xunit=0.01in, yunit=0.01in}
\psset{linewidth=0.005in}
\begin{pspicture}(0,0)(100,100)
\psline[linecolor=darkgray,linestyle=solid,linewidth=0.02in](0,55)(105,55)
\psaxes[dy=10,Dy=1,dx=10,Dx=1,tickstyle=bottom,showorigin=false,labels=none](0,0)(100,100)
\pscircle*(10,10){0.04in}
\pscircle*(30,20){0.04in}
\pscircle*(50,30){0.04in}
\pscircle*(70,40){0.04in}
\pscircle*(90,50){0.04in}
\pscircle*(100,60){0.04in}
\pscircle*(80,70){0.04in}
\pscircle*(60,80){0.04in}
\pscircle*(40,90){0.04in}
\pscircle*(20,100){0.04in}
\end{pspicture}
\end{tabular}
\end{center}
\caption{From left to right, an arbitrary horizontal alternation, a horizontal parallel alternation, and a vertical wedge alternation.}\label{fig-alternation}
\end{figure}
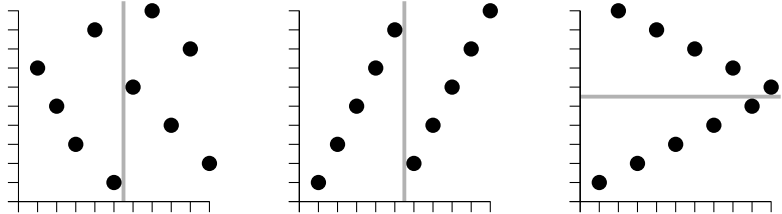

An {\it alternation\/} is a permutation in which every odd entry lies to the left of every even entry, or any symmetry of such a permutation.%
\footnote{\label{fn-alternations}It is sometimes necessary to specify the direction in which this alternation occurs.  In such cases, we use the term {\it horizontal alternation\/} to refer to a permutation in which every odd entry lies to the left of every even entry, or the reverse of such a permutation (in which case the entries of the permutation lie alternatively to the left and then to the right of a vertical line, e.g., the permutations shown on the left and center of Figure~\ref{fig-alternation}). A vertical alternation is the group-theoretic inverse of a horizontal alternation (e.g., the permutation shown on the right of Figure~\ref{fig-alternation} is a vertical alternation).}
Equivalently, an alternation is a permutation whose plot can be divided into two parts, by a single horizontal or vertical line, so that for every pair of entries from the same part there is a entry from the other part which {\it separates\/} them, i.e., there is a entry from the other part which lies either horizontally or vertically between them.  A {\it parallel alternation\/} is an alternation in which these two sets of entries form monotone subsequences, either both increasing or both decreasing, while a {\it wedge alternation\/} is one in which the two sets of entries form monotone subsequences pointing in opposite directions.  Note that every parallel alternation is either simple or nearly so in the sense that at most two entries need be removed from any parallel alternation to obtain a simple permutation.  On the contrary, wedge alternations are not simple (although they can be made simple by the addition of at most one entry).  Our next result follows routinely from the Erd\H{o}s-Szekeres Theorem.

\begin{proposition}\label{prop-alt-par-wedge}
If the permutation class $\C$ contains arbitrarily long alternations, then it contains arbitrarily long wedge or parallel alternations.
\end{proposition}
\begin{proof}
Suppose that $\C$ contains an alternation of length $2n\ge 2k^4$.  By symmetry, we may assume that this alternation is a vertical alternation.  The Erd\H{o}s-Szekeres Theorem then implies that the sequence $\pi(1),\pi(3),\dots,\pi(2n-1)$ contains a monotone subsequence of length at least $k^2$, say $\pi(i_1),\pi(i_2),\dots,\pi(i_{k^2})$.  Applying the Erd\H{o}s-Szekeres Theorem to the subsequence $\pi(i_1+1),\pi(i_2+1),\dots,\pi(i_{k^2}+1)$ completes the proof.
\end{proof}

The following result of Schmerl and Trotter is useful for inductive proofs about simple permutations.  While we state only the permutation case of this result (a proof of which is also given by Murphy~\cite{murphy:restricted-perm:}), Schmerl and Trotter's proof includes all irreflexive binary relational structures.  A generalization to ``$k$-structures'' is given by Ehrenfeucht and McConnell~\cite{ehrenfeucht:a-k-structure-g:}.

\begin{theorem}[Schmerl and Trotter~\cite{schmerl:critically-inde:}]\label{thm-schmerl-trotter}
Every simple permutation $\pi$ of length $n\ge 2$ contains a simple permutation of length $n-1$ or $n-2$.  Furthermore, $\pi$ contains a simple permutation of length $n-1$ unless $\pi$ is a simple parallel alternation (recall the central permutation in Figure~\ref{fig-alternation}).
\end{theorem}

\minisec{Direct sums, skew sums, and sum completions}
Two inflations in particular occur frequently enough that they deserve their own names: the {\it direct sum\/} $\pi\oplus\sigma=12[\pi,\sigma]$ and the {\it skew sum\/} $\pi\ominus\sigma=21[\pi,\sigma]$.  A class $\C$ is said to be {\it sum complete\/} if $\pi\oplus\sigma\in\C$ whenever $\pi,\sigma\in\C$.  Analogously, a class is said to be {\it skew sum complete\/} if $\pi\ominus\sigma\in\C$ for all $\pi,\sigma\in\C$.  Furthermore, a permutation is {\it sum indecomposable\/} (or {\it connected\/}) if it cannot be expressed as the direct sum of two shorter permutations and {\it skew sum indecomposable\/} (or simply {\it skew decomposable\/}) if it cannot be expressed as the skew sum of two shorter permutations.

Recall that the {\it permutation graph\/} corresponding to $\pi$ of length $n$ is the graph $G_\pi$ with vertices labeled by $[n]$, where $i\sim j$ if $i<j$ and $\pi(i)>\pi(j)$, i.e., if the entries in positions $i$ and $j$ form an inversion.  It is easy to see that this graph captures the notion of sum indecomposability:

\begin{proposition}\label{sum-indecomp-connected}
The permutation $\pi$ is sum indecomposable if and only if $G_\pi$ is connected.
\end{proposition}
\begin{proof}
Suppose $\pi$ has length $n$.  Clearly the vertices of any connected component of $G_\pi$ must be labeled by an interval $[i,j]$.  Thus if $G_\pi$ if disconnected, it contains a connected component of the form $[1,j]$ for some $j$, from which it follows that $\pi=\pi(1)\cdots\pi(j)\oplus\pi(j+1)\cdots\pi(n)$.  On the other hand, if $\pi=\sigma\oplus\tau$ then $G_\pi$ is the disjoint union of $G_\sigma$ and $G_\tau$, and thus disconnected.
\end{proof}

Given a set of permutations $X$ we define the {\it sum completion of $X$\/}, denoted by $\bigoplus X$, to be the smallest sum complete permutation class containing $X$.  Equivalently,
$$
\bigoplus X=\{\pi_1\oplus\pi_2\oplus\cdots\oplus\pi_k : \mbox{each $\pi_i$ is contained in an element of $X$}\}.
$$
When $X$ is a singleton, say $X=\{\pi\}$, we simply write $\bigoplus\pi$ for its sum completion.  We analogously define the {\it skew sum completion of $X$\/} to be the smallest skew sum complete permutation class containing $X$.  Counting sum complete classes is easy, given enough information about the sum indecomposables:

\begin{proposition}\label{enum-oplus-completion}
Let $f$ denote the generating function for the set of sum indecomposable permutations contained in members of $X$ (not counting the empty permutation).  Then the generating function for $\bigoplus X$ is $1/(1-f)$.%
%
%
\end{proposition}
\begin{proof}
There is a canonical bijection between elements of $\bigoplus X$ and sequences of nonempty sum indecomposable permutations contained in members of $X$.  Therefore the generating function for $\bigoplus X$ is $1+f+f^2+\cdots$, establishing the proposition.
\end{proof}

As our interest lies in growth rates rather than exact enumeration, we typically follow a use of Proposition~\ref{enum-oplus-completion} with an application of Pringsheim's Theorem:

\newtheorem*{pringsheimsthm}{Pringsheim's Theorem}
\begin{pringsheimsthm}[see Flajolet and Sedgewick~{\cite[Section IV.3]{flajolet:analytic-combin:}}]
The upper growth rate, $\limsup_{n\rightarrow\infty}\sqrt[n]{a_n}$, of a sequence $(a_n)_{n\ge 0}$ of nonnegative numbers is equal to the reciprocal of the smallest positive singularity of the power series $\sum a_nx^n$.
\end{pringsheimsthm}

We conclude our discussion of sums and skew sums with the following result of Arratia.

\begin{proposition}[Arratia~\cite{arratia:on-the-stanley-:}]\label{arratia-gr}
Every sum or skew sum complete permutation class has a (possibly infinite) growth rate.
\end{proposition}
\begin{proof}
Suppose, without loss of generality that $\C$ is a sum complete permutation class.  Then $\oplus$ gives an injection from $\C_m\times\C_n$ to $\C_{m+n}$, so the sequence $|\C_n|$ is supermultiplicative and thus $\lim_{n\rightarrow\infty} \sqrt[n]{|\C_n|}$ exists by Fekete's Lemma%
\footnote{Fekete's Lemma in its more typical form says that if a $a_n$ is superadditive, meaning that $a_{m+n}\ge a_m+a_n$, then $\lim_{n\rightarrow\infty} a_n/n$ exists and is equal to $\sup a_n/n$.  To apply this form of Fekete's Lemma in our context, consider the sequence $\log |\C_n|$.}.
\end{proof}

\minisec{The increasing oscillating sequence}
Intimately connected with sum indecomposability is the {\it increasing oscillating sequence\/}\footnote{This sequence, listed as \OEISlink{A065164} in the \OEISref, also arises in the study of juggling, see Buhler, Eisenbud, Graham, and Wright~\cite{buhler:juggling-drops-:}.},
$$
4,1,6,3,8,5,\dots,2k+2,2k-1,\dots;
$$
a plot is shown in Figure~\ref{fig-inc-osc}.  We further define an {\it increasing oscillation\/}\footnote{Increasing oscillations are called {\it Gollan permutations\/} in computational biology, because they are the unique permutations of length $n$ which require $n-1$ steps to sort by reversals, see Pevzner~\cite{pevzner:computational-m:}.} to be any sum indecomposable permutation that is contained in the increasing oscillation, a {\it decreasing oscillation\/} to be the reverse of an increasing oscillation, and an {\it oscillation\/} to be any permutation that is either an increasing or a decreasing oscillation (this term dates back to at least Pratt~\cite{pratt:computing-permu:}).  Note that for $k\neq 3$, all oscillations of length $k$ are simple permutations.

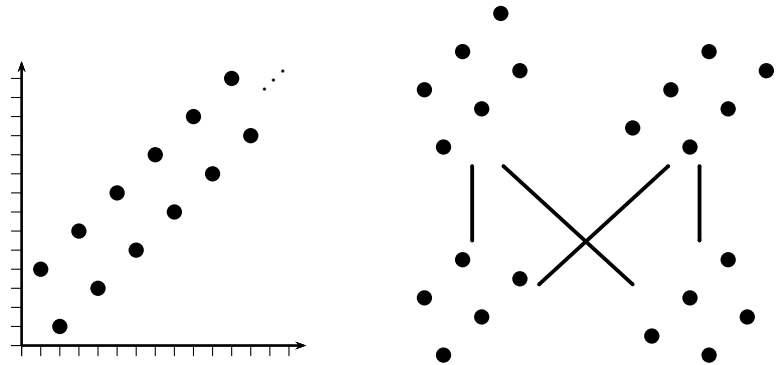
\begin{figure}
\begin{center}
\begin{tabular}{ccc}
\psset{xunit=0.01in, yunit=0.01in}
\psset{linewidth=1\psxunit}
\begin{pspicture}(0,0)(150,150)
\psaxes[dy=10, Dy=1, dx=10, Dx=1,tickstyle=bottom,showorigin=false,labels=none]{->}(0,0)(149,149)
\pscircle*(10,40){0.04in}
\pscircle*(20,10){0.04in}
\pscircle*(30,60){0.04in}
\pscircle*(40,30){0.04in}
\pscircle*(50,80){0.04in}
\pscircle*(60,50){0.04in}
\pscircle*(70,100){0.04in}
\pscircle*(80,70){0.04in}
\pscircle*(90,120){0.04in}
\pscircle*(100,90){0.04in}
\pscircle*(110,140){0.04in}
\pscircle*(120,110){0.04in}
\pstextpath[c]{\psline[linecolor=white](125,138)(135,148)}{$\dots$}
\end{pspicture}
&
\rule{20pt}{0pt}
&
\psset{xunit=0.01in, yunit=0.01in}
\psset{linewidth=0.02in}
\begin{psmatrix}
	\psset{xunit=0.01in, yunit=0.01in}
	\begin{pspicture}(10,10)(60,80)
	\pscircle*(10,40){0.04in}
	\pscircle*(20,10){0.04in}
	\pscircle*(30,60){0.04in}
	\pscircle*(40,30){0.04in}
	\pscircle*(50,80){0.04in}
	\pscircle*(60,50){0.04in}
	\end{pspicture}
&
	\psset{xunit=0.01in, yunit=0.01in}
	\begin{pspicture}(10,30)(80,80)
	\pscircle*(10,40){0.04in}
	\pscircle*(30,60){0.04in}
	\pscircle*(40,30){0.04in}
	\pscircle*(50,80){0.04in}
	\pscircle*(60,50){0.04in}
	\pscircle*(80,70){0.04in}
	\end{pspicture}
\\
	\psset{xunit=0.01in, yunit=0.01in}
	\begin{pspicture}(10,10)(60,60)
	\pscircle*(10,40){0.04in}
	\pscircle*(20,10){0.04in}
	\pscircle*(30,60){0.04in}
	\pscircle*(40,30){0.04in}
	\pscircle*(60,50){0.04in}
	\end{pspicture}
&
	\psset{xunit=0.01in, yunit=0.01in}
	\begin{pspicture}(10,30)(60,80)
	\pscircle*(10,40){0.04in}
	\pscircle*(30,60){0.04in}
	\pscircle*(40,30){0.04in}
	\pscircle*(50,80){0.04in}
	\pscircle*(60,50){0.04in}
	\end{pspicture}
\ncline[nodesep=0.1in]{c-c}{1,1}{2,1}
\ncline[nodesep=0.1in]{c-c}{1,1}{2,2}
\ncline[nodesep=0.1in]{c-c}{1,2}{2,2}
\ncline[nodesep=0.1in]{c-c}{1,2}{2,1}
\end{psmatrix}
\end{tabular}
\end{center}
\caption{A plot of the increasing oscillating sequence (left) and a selection of the Hasse diagram of increasing oscillations (right).}
\label{fig-inc-osc}
\end{figure}

The next proposition helps explain the connection between sum indecomposability and increasing oscillations.

\begin{proposition}\label{inc-osc-path}
If $G_\pi$ is an induced path then $\pi$ is an increasing oscillation.
\end{proposition}
\begin{proof}
Suppose that $G_\pi$ is an induced path and label its vertices $\{p_1\sim p_2\sim\cdots\}$.  Note that $\pi$ is an increasing oscillation if and only if its inverse is, and thus we assume by symmetry that $p_2$ lies below and to the right of $p_1$.   The third vertex, $p_3$ (if it exists) must then be adjacent to $p_2$ but not to $p_1$ so must lie either
\begin{itemize}
\item horizontally between $p_1$ and $p_2$ and above $p_1$ or
\item vertical between $p_1$ and $p_2$ and to the left of $p_1$.
\end{itemize}
We consider only the first of these two cases, as the other follows similarly.  The fourth vertex $p_4$ (again, if it exists) must lie vertically between $p_1$ and $p_3$ and to the right of $p_2$.  Continuing in this manner it is easy to see that $\pi$ is contained in the increasing oscillating sequence, and $\pi$ is sum indecomposable because $G_\pi$ is connected (Proposition~\ref{sum-indecomp-connected}).
\end{proof}

With Proposition~\ref{inc-osc-path} in hand, we can slightly strengthen Proposition~\ref{sum-indecomp-connected}.

\begin{proposition}\label{sum-indecomp-connected-strong}
The permutation $\pi$ of length $n$ is sum indecomposable if and only if $G_\pi$ contains a path connecting the vertices $1$ and $n$.
\end{proposition}
\begin{proof}
If $\pi$ is sum indecomposable then $G_\pi$ is connected by Proposition~\ref{sum-indecomp-connected}, it contains the desired path.  Now suppose that $G_\pi$ contains a path connecting the vertices $1$ and $n$.  Let $P=\{1=p_1\sim p_2\sim\cdots\sim p_m=n\}$ denote a shortest path between the vertices $1$ and $n$, so $P$ is an induced path.  By Proposition~\ref{inc-osc-path}, the points $p_1,p_2,\dots,p_m$ are order isomorphic to an increasing oscillation.

We then have that every entry above and to the left of a $p_{2\ell}$ entry is adjacent to $p_{2\ell}$ in $G_\pi$ while every entry below and to the right of a $p_{2\ell-1}$ is adjacent to $p_{2\ell-1}$ in $G_\pi$.  It is easy to check that at least one of these conditions holds for each entry of $\pi$, establishing that $G_\pi$ is connected, as desired.
\end{proof}

In addition to their connections to sum indecomposable permutations, we also need facts about increasing oscillations themselves.  We begin by recalling the basis for the class of permutations contained in some increasing oscillation.  This result was stated without proof in Murphy's thesis~\cite{murphy:restricted-perm:}, but has since been proved formally.

\begin{proposition}[Brignall, Ru\v{s}kuc, and Vatter~\cite{brignall:simple-permutat:b}]\label{prop-embed-inc-osc}
The class of all permutations contained in some increasing oscillation is $\Av(321,2341,3412,4123)$.
\end{proposition}

The proof of Proposition~\ref{prop-embed-inc-osc} in \cite{brignall:simple-permutat:b} uses an encoding of permutations known as the ``rank encoding\footnote{We refer the reader to Albert, Atkinson, and Ru\v{s}kuc~\cite{albert:regular-closed-:} for a detailed study of the rank encoding.}.''  The following result follows immediately from that work.

\begin{proposition}\label{prop-growth-osc}
The generating function for the class of all permutations contained in some increasing oscillation is $(1-x)/(1-2x-x^3)$, and thus its growth rate is $\kappa$.
\end{proposition}

\minisec{Partial well-order}
The poset $(P,\le)$ is said to be {\it partially well-ordered (pwo)\/} if it contains neither an infinite descending chain nor an infinite antichain.  The results stated here have been rediscovered many times and so we will not attempt attribution.

\begin{proposition}\label{pwo-tfae}
The poset $(P,\le)$ without infinite descending chains --- and thus in particular, any permutation class --- is pwo if and only if every infinite sequence of elements of $P$ contains an infinite ascending sequence.
\end{proposition}

Given a poset $(P,\le)$, we denote the set of words (or, sequences) with letters from $P$ as $P^*$.  There are a variety of partial orders on $P^*$, but for our purposes the most important is the {\it subword order\/} (or, {\it scattered subword order\/} or {\it subsequence order\/}) defined by $v\le w$ if there are indices $1\le i_1<i_2<\cdots<i_k=|v|$ such that $v_j\le w_{i_j}$ for all $j\in[k]$.  In this context we have the following result, originally due to Higman~\cite{higman:ordering-by-div:} but rediscovered many times since.

\newtheorem*{higmans-theorem}{Higman's Theorem}
\begin{higmans-theorem}
If $(P,\le)$ is pwo then $P^*$, ordered by the subword order, is also pwo.
\end{higmans-theorem}

Similarly, given a list of posets $(P_1,\le_1),\dots,(P_s,\le_s)$, the product order $(P_1,\le_1)\times\cdots\times(P_s,\le_s)$ is the poset containing the tuples $P_1\times\cdots\times P_s$, equipped with the order $(x_1,\dots,x_s)\le (y_1,\dots,y_s)$ if and only if $x_i\le_i y_i$ for all $i\in[s]$.  As an immediate corollary of Higman's Theorem, we obtain the following.

\begin{proposition}\label{product-pwo}
The product $(P_1,\le_1)\times\cdots\times(P_s,\le_s)$ of a collection of posets is pwo if and only if each of them is pwo.
\end{proposition}

We now specialize to permutation classes.

\begin{proposition}\label{pwo-subclasses-dcc}
The subclasses of a pwo permutation class satisfy the {\it descending chain condition\/}, i.e., if $\C$ is a pwo class, there does not exist a sequence $\C=\C^1\supsetneq \C^2\supsetneq \C^3\supsetneq\cdots$ of permutation classes, or in other words, the subclasses of $\C$ are well-founded under the subset order.
\end{proposition}
\begin{proof}
Suppose to the contrary that $\C$ contains an infinite strictly decreasing sequence $\C=\C^1\supsetneq \C^2\supsetneq \C^3\supsetneq\cdots$ and for each $i$ choose $\beta_i\in\C^i\setminus\C^{i+1}$.  As each $\beta_i$ lies in the pwo class $\C$, Proposition~\ref{pwo-tfae} shows that the sequence $\beta_1,\beta_2,\dots$ contains an ascent, say $\beta_i\le \beta_j$ for some indices $i<j$.  This, however, shows that $\beta_i\in\C^j$, contradicting our choice of $\beta_i$.
\end{proof}

\begin{proposition}\label{pwo-subclasses}
A pwo permutation class contains only countably many subclasses.
\end{proposition}
\begin{proof}
Let $\C$ be a pwo permutation class.  Every subclass $\D\subseteq\C$ is of the form $\C\cap\Av(B)$ for some antichain $B\subseteq\C$.  Since $\C$ is pwo, all such antichains are finite and thus the set of subclasses of $\C$ is countable.
\end{proof}

Of particular importance to us is the connection between simple permutations and pwo permutation classes:

\begin{proposition}[see Albert and Atkinson~\cite{albert:simple-permutat:} or Brignall~\cite{brignall:a-survey-of-sim:}]\label{fin-simples-pwo}
Every permutation class containing only finitely many simple permutations is pwo.
\end{proposition}

\begin{figure}
\begin{center}
\begin{tabular}{ccc}
\psset{xunit=0.01in, yunit=0.01in}
\psset{linewidth=0.005in}
\begin{pspicture}(0,0)(130,130)
\psaxes[dy=10,Dy=1,dx=10,Dx=1,tickstyle=bottom,showorigin=false,labels=none](0,0)(130,130)
\psframe[linecolor=darkgray,fillstyle=solid,fillcolor=lightgray,linewidth=0.02in](7,17)(23,33)
\psframe[linecolor=darkgray,fillstyle=solid,fillcolor=lightgray,linewidth=0.02in](107,117)(123,133)
\pscircle*(10,20){0.04in}
\pscircle*(20,30){0.04in}
\pscircle*(30,50){0.04in}
\pscircle*(40,10){0.04in}
\pscircle*(50,70){0.04in}
\pscircle*(60,40){0.04in}
\pscircle*(70,90){0.04in}
\pscircle*(80,60){0.04in}
\pscircle*(90,110){0.04in}
\pscircle*(100,80){0.04in}
\pscircle*(110,120){0.04in}
\pscircle*(120,130){0.04in}
\pscircle*(130,100){0.04in}
\end{pspicture}
&
\rule{10pt}{0pt}
&
\psset{xunit=0.01in, yunit=0.01in}
\psset{linewidth=0.02in}
\begin{pspicture}(0,0)(130,130)
\pscircle*(10,20){0.04in}
\pscircle*(20,30){0.04in}
\pscircle*(30,50){0.04in}
\pscircle*(40,10){0.04in}
\pscircle*(50,70){0.04in}
\pscircle*(60,40){0.04in}
\pscircle*(70,90){0.04in}
\pscircle*(80,60){0.04in}
\pscircle*(90,110){0.04in}
\pscircle*(100,80){0.04in}
\pscircle*(110,120){0.04in}
\pscircle*(120,130){0.04in}
\pscircle*(130,100){0.04in}
\psline(10,20)(40,10)
\psline(20,30)(40,10)
\psline(30,50)(40,10)
\psline(30,50)(60,40)
\psline(50,70)(60,40)
\psline(50,70)(80,60)
\psline(70,90)(80,60)
\psline(70,90)(100,80)
\psline(90,110)(100,80)
\psline(90,110)(130,100)
\psline(110,120)(130,100)
\psline(120,130)(130,100)
\end{pspicture}
\end{tabular}
\end{center}
\caption{The plot (left) and permutation graph (right) of $u_4$.}\label{fig-u4}.
\end{figure}
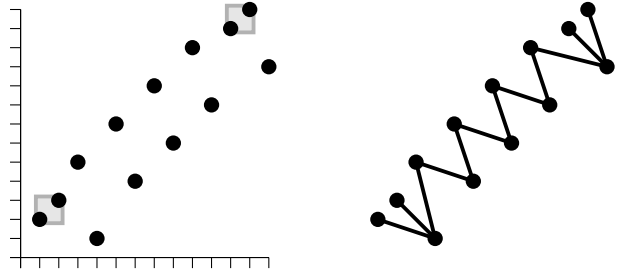

One of the most straightforward infinite antichains of permutations to describe is $U=\{u_1,u_2,\dots\}$ defined by
\begin{eqnarray*}
u_1&=&2,3,5,1\ |\ |\ 6,7,4\\
u_2&=&2,3,5,1\ |\ 7,4\ |\ 8,9,6\\
u_3&=&2,3,5,1\ |\ 7,4,9,6\ |\ 10,11,8\\
&\vdots&\\
u_k&=&2,3,5,1\ |\ 7,4,9,6,11,8,\dots,2k+3,2k\ |\ 2k+4,2k+5,2k+2\\
&\vdots&
\end{eqnarray*}
Here the vertical bars have no mathematical meaning but are meant only to emphasize the different parts of the permutations.  Note that the antichain $U$ is formed by inflating the first and greatest elements of every odd-length increasing oscillation, see Figure~\ref{fig-u4}.  Many variations on this theme are possible, see Vatter~\cite{vatter:permutation-cla}.

\begin{proposition}\label{U-antichain}
The set $U$ forms an infinite antichain of permutations.
\end{proposition}
\newenvironment{U-antichain-proof}{\medskip\noindent {\it Proof, due to Klazar~\cite{klazar:on-the-least-ex:}.\/}}{\qed\bigskip}
\begin{U-antichain-proof}
Clearly the permutation graph $G_\sigma$ must be contained, as an induced subgraph, in $G_\pi$ whenever $\sigma\le\pi$, and almost as clearly, the set of permutation graphs $\{G_{u_1},G_{u_2},\dots\}$ forms an infinite antichain under the induced subgraph order.
\end{U-antichain-proof}

This antichain leads to an upper bound for Answer 2, but first we must make the following observation.

\begin{proposition}\label{non-pwo-uncountable}
Every permutation class containing an infinite antichain contains uncountably many subclasses.
\end{proposition}
\begin{proof}
Suppose the class $\C$ contains the infinite antichain $A$.  Then every class of the form $\C\cap\Av(B)$, $B\subseteq A$ is distinct.
\end{proof}

Now note that
$$
U\subseteq \Av(321,3412,4123,23451,134526,134625,314526,314625).
$$
The Maple package {\sc InsEnc}, described in Vatter~\cite{vatter:finding-regular:}, computes the generating function of this class as
$$
\frac{x(1+x+x^2+2x^3+3x^4+3x^5+x^6-x^7-x^9)}{(1+x)(1-2x-x^3)},
$$
which shows (via Pringsheim's Theorem) that its growth rate is $\kappa$.  Therefore there are uncountably many permutation classes with upper growth rate at most $\kappa$.

\minisec{Atomicity}
A permutation class is said to be {\it atomic\/} if it cannot be expressed as the union of two proper subclasses.  Like partial well-order, atomicity (under a variety of names) has been rediscovered numerous times; it seems to date originally to a 1954 article of Fra{\"{\i}}ss{\'e}~\cite{fraisse:sur-lextension-:}.  Murphy undertook a particularly thorough investigation of atomic permutation classes in his thesis~\cite{murphy:restricted-perm:}, and all of the results here can be found there.

It is not difficult to show that the {\it joint embedding property\/} is a necessary and sufficient condition for the permutation class $\C$ to be atomic; this condition states that for all $\pi,\sigma\in\C$, there is a $\tau\in\C$ containing both $\pi$ and $\sigma$.  Fra{\"{\i}}ss{\'e} established another necessary and sufficient condition for atomicity, which we describe only in the permutation context (Fra{\"{\i}}ss{\'e} proved his results for relational structures).  Given two linearly ordered sets (or simply, linear orders) $A$ and $B$ and a bijection $f:A\rightarrow B$, every finite subset $\{a_1<\cdots<a_n\}\subseteq A$ maps to a finite sequence $f(a_1),\dots,f(a_n)\in B$ that is order isomorphic to a unique permutation.  We call the set of permutations that arise in this manner the {\it age of $f$\/}, denoted $\Age(f:A\rightarrow B)$.

\begin{theorem}[Fra{\"{\i}}ss{\'e}~\cite{fraisse:sur-lextension-:}; see also Hodges~{\cite[Section 7.1]{hodges:model-theory:}}]\label{atomic-tfae}
For a permutation class $\C$, the following are equivalent:
\begin{enumerate}
\item[(1)] $\C$ is atomic,
\item[(2)] $\C$ satisfies the joint embedding property, and
\item[(3)] $\C=\Age(f:A\rightarrow B)$ for a bijection $f$ between two countable linear orders $A$ and $B$.
\end{enumerate}
\end{theorem}

We use only the generic properties of atomic classes, i.e., those that hold for any appropriate type of object, but note that atomic classes of permutations are particularly interesting%
\footnote{\newcommand{\nikt}{\mathcal{T}}
For instance, define $\nikt(X,Y)$ as the set of all (necessarily atomic) classes of permutations that can be expressed as $\Age(f:X\rightarrow Y)$.  The following two questions then naturally arise:
\begin{itemize}
\item For a given $X$ and $Y$, can one characterize $\nikt(X,Y)$?  Can it be decided whether a given class lies in $\nikt(X,Y)$?
\item For what linear orders $X$, $Y$, $W$, and $Z$ is $\nikt(X,Y)\subseteq\nikt(W,Z)$?
\end{itemize}
For some partial answers to these questions, the reader is referred to Atkinson, Murphy, and Ru\v{s}kuc~\cite{atkinson:pattern-avoidan:} and Huczynska and Ru\v{s}kuc~\cite{huczynska:pattern-classes:}.}.

In addition to Theorem~\ref{atomic-tfae}, we need several results about atomicity for pwo classes.  For the first, we follow the proof given by Murphy~\cite{murphy:restricted-perm:}.


\begin{proposition}\label{pwo-atomic-union}
Every pwo permutation class can be expressed as a finite union of atomic classes.
\end{proposition}
\begin{proof}
Consider the binary tree whose root is the pwo class $\C$, all of whose leaves are atomic classes, and in which the children of the non-atomic class $\D$ are two proper subclasses $\D^1,\D^2\subsetneq\D$ such that $\D^1\cup\D^2=\D$.  Because $\C$ is pwo its subclasses satisfy by the descending chain condition by Proposition~\ref{pwo-subclasses-dcc}, so this tree contains no infinite paths and thus is finite by K\"onig's Lemma; its leaves give the desired atomic classes.
\end{proof}

The problem of computing growth rates of pwo classes can then be reduced to that of computing growth rates of atomic classes:

\begin{proposition}\label{pwo-atomic-gr}
For a pwo permutation class $\C$, $\ugr(\C)$ is equal to the maximum upper growth rate of an atomic subclass of $\C$.
\end{proposition}
\begin{proof}
Using Proposition~\ref{pwo-atomic-union}, write $\C$ as a union of finitely many atomic subclasses, $\C=\C^1\cup\cdots\cup\C^m$, and then choose an infinite subsequence $n_1<n_2<\cdots$ such that $\ugr(\C)=\lim\sqrt[n_i]{|\C_{n_i}|}$.  For each $n$, at least $1/m$ of the permutations in $\C_n$ lie in a particular $\C^j$, and thus there is an infinite subsequence of the $n_i$s, say $n_1'<n_2'<\cdots$, such that at least $1/m$ of the permutations in $\C_{n_i'}$ lie in the same $\C^j_{n_i'}$ for all $i$.  Then
$$
\ugr(\C)=\lim_{i\rightarrow\infty}\sqrt[n_i]{|\C_{n_i}|}=\lim_{i\rightarrow\infty}\sqrt[n_i']{|\C_{n_i'}|}
\le
\limsup_{i\rightarrow\infty}\sqrt[n_i']{m|\C^j_{n_i'}|}
\le
\ugr(\C^j).
$$
The reverse inequality is obvious, proving the proposition.
\end{proof}

\section{Generalized Grid Classes}\label{sec-generalized-grids}

When discussing specific grid classes (which will rarely be necessary), we index matrices beginning from the lower left-hand corner and we reverse the rows and columns, so $\M_{3,2}$ denotes for us the entry of $\M$ in the $3$rd column from the left and $2$nd row from the bottom.  Below we include a $3\times 2$ matrix with its entries labeled:
$$
\left(
\begin{footnotesize}
\begin{array}{rrr}
(1,2)&(2,2)&(3,2)\\
(1,1)&(2,1)&(3,1)
\end{array}
\end{footnotesize}
\right).
$$

Roughly, the grid class of a matrix $\M$ is the set of all permutations that can be divided into a finite number of blocks, each containing a subsequence of a prescribed form dictated by $\M$.  Before grid classes can be defined formally, there are some notational prerequisites to be covered.

Given a permutation $\pi$ of length $n$ and sets $X,Y\subseteq[n]$, we write $\pi(X\times Y)$ for the permutation that is order isomorphic to the subsequence of $\pi$ with indices from $X$ and values in $Y$.  For example, to compute $391867452([3,7]\times[2,6])$ we consider the subsequence of entries in indices $3$ through $7$, $18674$, which have values between $2$ and $6$; in this case the subsequence is $64$, so $391867452([3,7]\times[2,6])=21$.

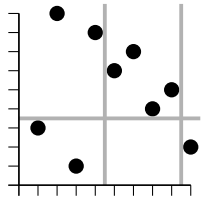
\begin{figure}
\begin{center}
\psset{xunit=0.01in, yunit=0.01in}
\psset{linewidth=0.005in}
\begin{pspicture}(0,0)(90,90)
\pscircle*(10,30){0.04in}
\pscircle*(20,90){0.04in}
\pscircle*(30,10){0.04in}
\pscircle*(40,80){0.04in}
\pscircle*(50,60){0.04in}
\pscircle*(60,70){0.04in}
\pscircle*(70,40){0.04in}
\pscircle*(80,50){0.04in}
\pscircle*(90,20){0.04in}
\psline[linecolor=darkgray,linestyle=solid,linewidth=0.02in](0,35)(95,35)
\psline[linecolor=darkgray,linestyle=solid,linewidth=0.02in](45,0)(45,95)
\psline[linecolor=darkgray,linestyle=solid,linewidth=0.02in](85,0)(85,95)
\psaxes[dy=10,Dy=1,dx=10,Dx=1,tickstyle=bottom,showorigin=false,labels=none](0,0)(90,90)
\end{pspicture}
\end{center}
\caption[]{An $\left(\begin{footnotesize}
\begin{array}{ccc}
\Av(12)&\Av(321)&\\
\Av(12)&&\Av(21)
\end{array}
\end{footnotesize}
\right)
$-gridding of the permutation $391867452$; the column divisions in the plot are $c_1,c_2,c_3,c_4=1,5,9,10$ and the row divisions are $r_1,r_2,r_3=1,4,10$.  Here and in what follows, we suppress $\emptyset$ entries from our matrices.}\label{fig-grid-391867452}
\end{figure}

Suppose that $\M$ is a $t\times u$ matrix (meaning, in our indexing, that $\M$ has $t$ columns and $u$ rows) whose entries are permutation classes.  An {\it $\M$-gridding\/} of the permutation $\pi$ of length $n$ is a pair of sequences $1=c_1\le\cdots\le c_{t+1}=n+1$ (the column divisions) and $1=r_1\le\cdots\le r_{u+1}=n+1$ (the row divisions) such that $\pi([c_k,c_{k+1})\times[r_\ell,r_{\ell+1}))$ is either empty or is a member of $\M_{k,\ell}$ for all $k\in[t]$ and $\ell\in[u]$.  Figure~\ref{fig-grid-391867452} shows an example.

The {\it grid class of $\M$\/}, written $\Grid(\M)$, consists of all permutations which possess an $\M$-gridding.  Furthermore, we say that the permutation class $\C$ is itself {\it $\M$-griddable\/} if $\C\subseteq\Grid(\M)$.  Our treatment of griddability is more general than earlier definitions from~\cite{huczynska:grid-classes-an:,murphy:profile-classes:,waton:on-permutation-:} which consider only {\it monotone grid classes\/}, defined as classes of the form $\Grid(\M)$ where each entry of $\M$ is $\Av(12)$, $\Av(21)$, or the empty class $\emptyset$.

We sometimes need to consider particular griddings of permutations, and in this case refer to a permutation together with an $\M$-gridding (if it has one) as an {\it $\M$-gridded permutation\/} (as opposed to an $\M$-griddable permutation); the set of all $\M$-gridded permutations may contain many different $\M$-griddings of the same permutation.  However, as demonstrated by the next proposition, this does not affect the rough asymptotics we are concerned with.

\begin{proposition}\label{gridded-gr}
For a matrix $\M$ of permutation classes and an $\M$-griddable class $\C$, the upper (resp., lower) growth rate of $\C$ is equal to the upper (resp., lower) growth rate of the sequence enumerating the $\M$-gridded permutations in $\C$.
\end{proposition}
\begin{proof}
Suppose that $\M$ is of size $t\times u$ and let $g_n$ denote the number of $\M$-gridded permutations of length $n$ in $\C$.  As every permutation in $\C$ has at least one $\M$-gridding ($\C$ is $\M$-griddable) and no permutation of length $n$ possesses more than ${n+t\choose t}{n+u\choose u}$ different $\M$-griddings (the column divisions form a multiset of size $t$ chosen from the set $[n+1]$, the row divisions, a multiset of size $u$) we get that
$$
g_n/{n+t\choose t}{n+u\choose u}\le |\C_n|\le g_n,
$$
from which the proposition immediately follows.
\end{proof}

We say that the class $\C$ is {\it $\{\D^1,\dots,\D^m\}$-griddable\/} if $\C$ is $\M$-griddable for some matrix $\M$ whose entries are all either empty or subclasses of some $\D^i$.  In the $m=1$ case we abbreviate ``$\{\D\}$-griddable'' to ``$\D$-griddable'', a situation which, as the following proposition shows, is no less general.

\begin{proposition}\label{grid-union}
A permutation class is $\{\D^1,\dots,\D^m\}$-griddable if and only if it is $\D^1\cup\cdots\cup\D^m$-griddable.
\end{proposition}
\begin{proof}
If the class $\C$ is $\{\D^1,\dots,\D^m\}$-griddable then it is clearly $\D^1\cup\cdots\cup\D^m$-griddable.  For the other direction, suppose that $\C$ is $\M$-griddable for a $t\times u$ matrix $\M$ whose entries are all (without loss of generality) equal to $\D^1\cup\cdots\cup\D^m$.  Thus for every cell $(j,k)$, the entries in cell $(j,k)$ in any $\M$-gridding of any permutation $\pi\in\C$ lie in one of the classes $\D^i$.  Therefore every $\pi\in\C$ is $\N$ griddable for the matrix $\N$ which consists of the direct sum\footnote{The {\it direct sum\/} of the matrices $A$ and $B$ is defined, in our indexing system, as $\left(\begin{array}{cc}0&B\\ A&0\end{array}\right)$.}
of every one of the (finitely many) $t\times u$ matrices with entries from $\{\D^1,\dots,\D^m\}$.
\end{proof}

It is useful to have several different perspectives of griddability, which require a bit more notation.  Given a permutation class $\D$, we say that the permutation $\pi$ of length $n$ can be {\it covered by $s$ $\D$-rectangles\/} if there are (not necessarily disjoint) rectangles $[w_1,x_1]\times[y_1,z_1]$,$\dots$,$[w_s,x_s]\times[y_s,z_s]\subseteq[n]\times[n]$ such that
\begin{itemize}
\item for each $i\in[s]$, $\pi([w_i,x_i]\times[y_i,z_i])\in\D$, and
\item for each $i\in[n]$, the point $(i,\pi(i))$ lies in $\displaystyle\bigcup_{i\in[s]} [w_i,x_i]\times[y_i,z_i]$.
\end{itemize}

For the third perspective, we say that the line $L$ {\it slices\/} the rectangle $R$ if $L$ intersects the interior of $R$.  If $\R$ is a collection of rectangles and $\L$ a collection of lines, we say that $\L$ slices $\R$ if every rectangle in $\R$ is sliced by some line from $\L$.  The rectangles we are interested in are always {\it axis-parallel\/}, meaning that each of their sides is parallel either to the $x$- or $y$-axis.  (Others use the term {\it axes-aligned\/} for such rectangles.)

\begin{proposition}\label{grid-coverings}
For permutation classes $\C$ and $\D$ the following are equivalent:
\begin{enumerate}
\item[(1)] $\C$ is $\D$-griddable,
\item[(2)] there is a constant $\ell$ so that for every permutation $\pi\in\C$, the set
$$
\{\mbox{axis-parallel rectangles $R$} : \pi(R)\not\in\D\}
$$
can be sliced by a collection of $\ell$ horizontal and vertical lines, and
\item[(3)] there is a constant $s$ so that every permutation in $\C$ can be covered by $s$ $\D$-rectangles.
\end{enumerate}
\end{proposition}
\begin{proof}
To begin with, if $\C$ is $\D$-griddable then $\C$ is $\M$-griddable for a matrix $\M$ of some size, say $t\times u$, whose entries consist of subclasses of $\D$.  Thus for every permutation $\pi\in\C_n$ there are column and row divisions $1=c_1\le\cdots\le c_{t+1}=n+1$ and $1=r_1\le\cdots\le r_{u+1}=n+1$ so that every subpermutation $\pi([c_k,c_{k+1})\times[r_\ell,r_{\ell+1}))$ lies in $\D$.  Therefore the corresponding lines, $x=c_1,\dots,x=c_t,y=r_1,\dots,y=r_u$, slice the given collection of rectangles, verifying that (1) implies (2).  That (2) implies (3) is similarly clear: any such collection of lines will slice the plane into a collection of rectangles which contain points order isomorphic to elements of $\D$.

This leaves us to establish that (3) implies (1).  Suppose that the permutation $\pi$ of length $n$ is covered by the $\D$-rectangles $[w_1,x_1]\times[y_1,z_1]$, $\dots$, $[w_s,x_s]\times[y_s,z_s]\subseteq[n]\times[n]$.  Define the indices $c_1,\dots,c_{2s}$ and $r_1,\dots,r_{2s}$ by
\begin{eqnarray*}
\{c_1\le\cdots\le c_{2s}\}&=&\{w_1,x_1,\dots,w_s,x_s\},\\
\{r_1\le\cdots\le r_{2s}\}&=&\{y_1,z_1,\dots,y_s,z_s\}.
\end{eqnarray*}
Since these rectangles cover $\pi$, we must have $c_1=r_1=1$ and $c_{2s}=r_{2s}=n$.  Now we claim that these sets of indices give a $(2s-1)\times (2s-1)$ $\D$-gridding of $\pi$.

To prove this claim it suffices to show that $\pi([c_k,c_{k+1}]\times[r_\ell,r_{\ell+1}])\in\D$ for every $k,\ell\in[2s-1]$.  Because the rectangles given cover $\pi$, the point $(c_k,r_\ell)$ lies in at least one rectangle, say $[w_m,x_m]\times[y_m,z_m]$.  Thus $c_k\ge w_m$ and $r_\ell\ge y_m$ and, because of the ordering of the $c$s and $r$s, we have $c_{k+1}\le x_m$ and $r_{\ell+1}\le z_m$.  Therefore $[c_k,c_{k+1}]\times[r_\ell,r_{\ell+1}]$ is contained in $[w_m,x_m]\times[y_m,z_m]$ and so $\pi([c_k,c_{k+1}]\times[r_\ell,r_{\ell+1}])\in\D$.

Therefore, if (3) holds, then the above argument implies that every permutation in $\C$ is $\M$-griddable for the $(2s-1)\times(2s-1)$ matrix $\M$ whose every entry is $\D$, and thus $\C$ is $\D$-griddable, as desired.
\end{proof}

This language makes the following already fairly obvious facts a bit easier to prove.

\begin{proposition}\label{grid-transitivity}
If $\C$ is $\D$-griddable and $\D$ is $\E$-griddable then $\C$ is $\E$-griddable.
\end{proposition}
\begin{proof}
Since $\C$ is $\D$-griddable, Proposition~\ref{grid-coverings} shows that there is a constant $s$ so that every permutation in $\C$ can be covered by $s$ $\D$-rectangles.  Similarly, there is a constant $t$ so that every permutation in $\D$, and thus every $\D$-rectangle, can be covered by $t$ $\E$-rectangles.  This shows that every permutation in $\C$ can be covered by $st$ $\E$-rectangles, which, by Proposition~\ref{grid-coverings}, establishes the proposition.
\end{proof}

\begin{proposition}\label{grid-intersection}
If $\C$ is both $\D$- and $\E$-griddable then $\C$ is $\D\cap\E$-griddable.
\end{proposition}
\begin{proof}
Since $\C$ is $\D$-griddable, Proposition~\ref{grid-coverings} shows that there is a constant $\ell$ such that for every permutation $\pi\in\C$, the set of axes parallel rectangles which are not order isomorphic to a permutation in $\D$ can be sliced by a collection of $\ell$ horizontal and vertical lines.  Similarly, there is a constant $k$ such that this holds for the axes parallel rectangles which are not order isomorphic to a permutation in $\E$.  Thus these $\ell+k$ lines taken together slice every rectangle which is not order isomorphic to a permutation in $\D\cap\E$.
\end{proof}

\begin{proposition}\label{grid-two-union}
If $\C$ is $\E$-griddable and $\D$ is $\mathcal{F}$-griddable, then $\C\cup\D$ is $\E\cup\mathcal{F}$-griddable.
\end{proposition}
\begin{proof}
Proposition~\ref{grid-coverings} shows that there is a constant $s$ so that every permutation in $\C$ can be covered by $s$ $\E$-rectangles while there is another constant $t$ so that every permutation in $\D$ can be covered by $t$ $\mathcal{F}$-rectangles.  Therefore every permutation in $\C\cup\D$ can be covered by $\max(s,t)$ $\E\cup\mathcal{F}$-rectangles.
\end{proof}

\begin{figure}
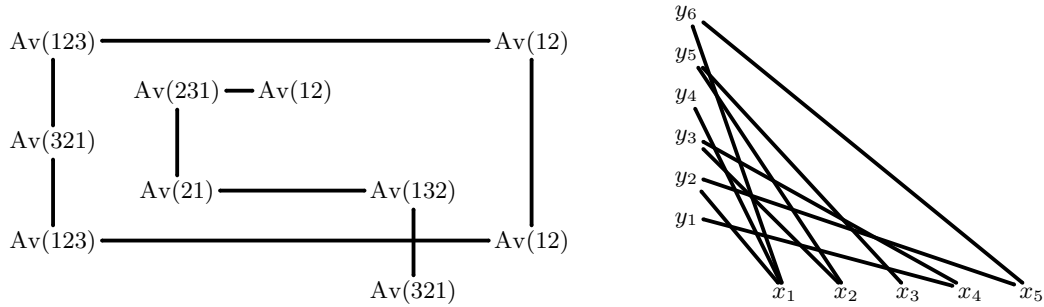

\begin{footnotesize}
$$
\begin{array}{ccc}
\begin{psmatrix}[rowsep=10pt,colsep=14pt,nodesep=2pt,linewidth=0.02in]
\Av(123)&&&&\Av(12)\\
&\Av(231)&\Av(12)&&\\
\Av(321)&&&&\\
&\Av(21)&&\Av(132)&\\
\Av(123)&&&&\Av(12)\\
&&&\Av(321)
\ncline{c-c}{1,1}{1,5}
\ncline{c-c}{1,1}{3,1}
\ncline{c-c}{3,1}{5,1}
\ncline{c-c}{5,1}{5,5}
\ncline{c-c}{5,5}{1,5}
\ncline{c-c}{2,2}{2,3}
\ncline{c-c}{2,2}{4,2}
\ncline{c-c}{4,2}{4,4}
\ncline{c-c}{4,4}{6,4}
\end{psmatrix}
&\rule{20pt}{0pt}&
\begin{psmatrix}[rowsep=10pt,colsep=14pt,nodesep=2pt,linewidth=0.02in]
y_6\\
y_5\\
y_4\\
y_3\\
y_2\\
y_1\\\\
&&x_1&x_2&x_3&x_4&x_5
\ncline{c-c}{1,1}{8,3}
\ncline{c-c}{1,1}{8,7}
\ncline{c-c}{2,1}{8,4}
\ncline{c-c}{2,1}{8,5}
\ncline{c-c}{3,1}{8,3}
\ncline{c-c}{4,1}{8,4}
\ncline{c-c}{4,1}{8,6}
\ncline{c-c}{5,1}{8,3}
\ncline{c-c}{5,1}{8,7}
\ncline{c-c}{6,1}{8,6}
\end{psmatrix}
\end{array}
$$
\end{footnotesize}
\caption{The cell graph (left) and row/column graph (right) of a matrix of permutation classes.}\label{fig-graph-grid}
\end{figure}

We conclude this section by discussing two different notions of the graph of a grid class.  First we define the {\it row/column graph\/} of the $t\times u$ matrix $\M$ of permutation classes as the bipartite graph on the vertices $x_1,\dots,x_t,y_1,\dots,y_u$ where $x_i\sim y_j$ if and only if $\M_{i,j}\neq\emptyset$.  The properties of this graph allow us to determine if a monotone grid class is partially well-ordered:

\begin{theorem}[Murphy and Vatter~\cite{murphy:profile-classes:}]\label{pwo-forest}
If $\M$ contains only monotone and empty classes then $\Grid(\M)$ is pwo if and only if the row/column graph of $\M$ is a forest.
\end{theorem}

A simpler proof of Theorem~\ref{pwo-forest} was given by Vatter and Waton~\cite{vatter:on-partial-well:}, while a generalization was proved by Brignall~\cite{brignall:grid-classes-an:}.

For our purposes, a different graph, called the {\it cell graph\/} of $\M$ is more useful.  This is the graph on the vertices $\{(i,j) : \M_{i,j}\neq\emptyset\}$ in which $(i,j)\sim(k,\ell)$ if $(i,j)$ and $(k,\ell)$ share either a row or a column, and there are no nonempty cells between them in this row or column.  Further, we label the vertex $(i,j)$ in this graph by the class it corresponds to, $\M_{i,j}$.  Figure~\ref{fig-graph-grid} shows the cell graph of a matrix.  Note that the distinction between cell and row/column graphs does not affect Theorem~\ref{pwo-forest}:

\begin{proposition}
Let $\M$ be a matrix of permutation classes.  The row/column graph of $\M$ is a forest if and only if the cell graph of $\M$ is a forest.
\end{proposition}
\begin{proof}
We begin by considering a cycle $x_{i_1}\sim y_{j_1}\sim \cdots\sim x_{i_k}\sim y_{j_k}\sim x_{i_1}$ in the row/column graph of $\M$ (as this graph is bipartite, such a cycle must be of even length and alternate between $x$ and $y$ vertices).  By definition, this means that $\M_{i_1,j_1},\dots,\M_{i_k,j_k},\M_{i_1,j_k}\neq\emptyset$.  This does not necessarily mean that $(i_1,j_1)\sim\cdots\sim(i_k,j_k)\sim(i_1,j_k)\sim(i_1,j_1)$ in the cell graph, because there may be nonempty cells in between these cells, but by including these cells we do find a cycle in the cell graph of $\M$.  In the other direction, such cells must similarly be removed, the details of which we omit.
\end{proof}

We also need a coarsening of the permutation containment order; for two $\M$-gridded permutations $\sigma$ and $\pi$ of respective lengths $k$ and $n$, we say that {\it $\pi$ contains a gridded copy of $\sigma$\/} and write $\sigma\le_g\pi$ if there are indices $1\le i_1<\cdots<i_k\le n$ such that the subsequence $\pi(i_1)\cdots\pi(i_k)$ is order isomorphic to $\sigma$ (this is the normal permutation containment order) and, further, that for every $j\in[k]$, $\sigma(j)$ and $\pi(i_j)$ lie in the same cell of $\M$ in the accompanying $\M$-griddings of $\sigma$ and $\pi$.  This is indeed coarser than the normal ordering because $\sigma\le_g\pi\implies\sigma\le\pi$.

Finally, we define a {\it connected component\/} of the matrix $\M$ of permutation classes to be a (necessarily rectangular) submatrix of $\M$ whose cells give rise to a connected component of the cell graph of $\M$.  Further, if $X$ is any subset of columns of $\M$ and $Y$ is any subset of rows of $\M$ then we write $\M(X\times Y)$ to denote the submatrix of $\M$ formed by the cells $X\times Y$.  For example, the matrix $\M$ whose cell graph is depicted in Figure~\ref{fig-graph-grid} contains two connected components:
$$
\M(\{1,5\}\times\{2,4,6\})=
\left(
\begin{footnotesize}
\begin{array}{cc}
\Av(123)&\Av(12)\\
\Av(321)&\\
\Av(123)&\Av(12)
\end{array}
\end{footnotesize}
\right)
$$
and
$$
\M(\{2,3,4\}\times\{1,3,5\})=
\left(
\begin{footnotesize}
\begin{array}{ccc}
\Av(231)&\Av(12)&\\
\Av(21)&&\Av(132)\\
&&\Av(321)
\end{array}
\end{footnotesize}
\right).
$$

\begin{proposition}\label{grid-pwo-component}
If the grid classes of each of its connected components are pwo then $\Grid(\M)$ is pwo.
\end{proposition}
\begin{proof}
Suppose that $\M$ has $s$ connected components, described by the columns $X_1,\dots,X_s$ and rows $Y_1,\dots,Y_s$.  Then there is a canonical order preserving bijection between the poset of all $\M$-gridded permutations, ordered by $\le_g$, and the poset of tuples $(\pi^1,\dots,\pi^s)$ where each $\pi^i$ is an $\M(X_i\times Y_i)$-gridded permutation, ordered by the product order $\le_g\times\cdots\times\le_g$.  The proof is then completed by Proposition~\ref{product-pwo}.
\end{proof}

We need a bit more notation for the last proposition of the section.  For an $\M$-gridded permutation $\pi$ and a subset $A=\{(j_1,k_1),\dots,(j_s,k_s)\}$ of cells of $\M$, we say that the (gridded) permutation formed by the entries of $\pi$ lying in the cells of $A$ is the {\it restriction of the gridded permutation $\pi$ to $A$\/}.  (Note that different griddings of $\pi$ will tend to lead to different restrictions.)  Furthermore, if $\C$ is a class, we say that {\it the restriction of $\C$ to $A$\/} is the set of all restrictions of its members to $A$.  Note that the restriction of a permutation class to a set of cells gives a set of gridded permutations which is closed downward under $\le_g$.

\begin{proposition}\label{grid-gr-component}
Suppose that $\C$ is $\M$-griddable.  The upper growth rate of $\C$ is the maximum of the upper growth rates of its restrictions to connected components of $\M$.
\end{proposition}
\begin{proof}
Let $g_n$ denote the number of $\M$-gridded permutations of length $n$ in $\C$, so the upper growth rate of $\C$ is equal to the upper growth rate of $g_n$ by Proposition~\ref{gridded-gr}.  Suppose that $\M$ has $s$ connected components, described by the columns $X_1,\dots,X_s$ and rows $Y_1,\dots,Y_s$, denote the restriction of $\C$ to $\M(X_i\times Y_i)$ by $\C^i$, and suppose that the greatest upper growth rate of any of these restrictions is $\gamma$.  It suffices to establish that the upper growth rate of $g_n$ is at most $\gamma$.

As remarked in the proof of Proposition~\ref{grid-pwo-component}, there is a canonical bijection between $\M$-gridded permutations and the poset of tuples $(\pi^1,\dots,\pi^s)$ where each $\pi^i$ is an $\M(X_i\times Y_i)$-gridded permutation.  Letting $g_{i,n}$ denote the number of $\M(X_i\times Y_i)$-gridded permutations of length $n$ in $\C^i$ we have
$$
g_n\le\sum_{n_1+\cdots+n_s=n} \prod_i g_{i,n_i}.
$$
Now fix $\epsilon>0$.  By our choice of $\gamma$, there is some $N$ such that for all $i\in[s]$ and $n_i>N$, $g_{i,n_i}<((1+\epsilon)\gamma)^{n_i}$.  There is also (trivially) an integer $N'\ge N$ such that for all $i\in[s]$, $n>N'$, and $n_i\le N$, $g_{i,n_i}\le (1+\epsilon)^n$.  Thus we have that for all $n>N'$ and $n_i\le n$,
$$
g_{i,n_i}
\le
\left\{\begin{array}{ll}
(1+\epsilon)^n\gamma^{n_i}&\mbox{if $n_i>N$ while}\\
(1+\epsilon)^n&\mbox{if $n_i\le N$.}
\end{array}
\right.
$$
It follows that for these values of $n$,
\begin{eqnarray*}
g_n
&\le&
{n+s-1\choose s-1}\max_{n_1+\cdots+n_s=n}\prod_i g_{i,n_i},\\
&\le&
{n+s-1\choose s-1}\max_{n_1+\cdots+n_s=n} \left(1+\epsilon\right)^{ns}\gamma^n,
\end{eqnarray*}
implying that $\displaystyle\limsup_{n\rightarrow\infty}\sqrt[n]{g_n}\le (1+\epsilon)^s\gamma$.  Letting $\epsilon\rightarrow 0$ completes the proof.
\end{proof}

To conclude this section we note that (upper, lower, proper) growth rates of specific grid classes can be computed using the method of Lagrange multipliers; for some details see Subsection~\ref{subsec-triple-alternations} of the Appendix.  Also, combining two recurring themes of this paper, Waton characterizes the atomic monotone grid classes in his thesis~\cite{waton:on-permutation-:}.  His characterization depends not only on row/column graphs, but also on the ``parity'' of the cycles of this graph.

\section{Characterizing $\D$-Griddable and Grid Irreducible Classes}\label{sec-gridding-characterization}


We begin this section with a characterization:

\begin{theorem}\label{gridding-characterization}
The permutation class $\C$ is $\D$-griddable if and only if it does not contain arbitrarily long sums or skew sums of basis elements of $\D$, that is, if there exists a constant $m$ so that $\C$ does not contain $\beta_1\oplus\cdots\oplus\beta_m$ or $\beta_1\ominus\cdots\ominus\beta_m$ for any basis elements $\beta_1,\dots,\beta_m$ of $\D$.

N.b. If $\D$ is finitely based then this condition can be simplified: $\C$ fails to have a $\D$-gridding if and only if $\C$ contains $\bigoplus\beta$ or $\bigominus\beta$ for some basis element $\beta$ of $\D$.
\end{theorem}

One direction of Theorem~\ref{gridding-characterization} is clear: if $\beta_1,\dots,\beta_m$ are basis elements of $\D$ then their direct sum $\beta_1\oplus\cdots\oplus\beta_m$ can be covered by no fewer than $m+1$ $\D$-rectangles, so if $\C$ contains arbitrarily long direct sums (equivalently, skew sums) of basis elements of $\D$ then it is not $\D$-griddable.

By the Proposition~\ref{grid-coverings} (2) interpretation of griddability, the other direction of Theorem~\ref{gridding-characterization} involves slicing a collection of rectangles in the plane --- in particular, the set $\{\mbox{axis-parallel rectangles $R$} : \pi(R)\not\in\D\}$ --- with a bounded number of vertical and horizontal lines.  Since we use these notions again in the next section, we cast this discussion in slightly more general terms.

We say that two rectangles $R,S$ are {\it independent\/} if both their $x$- and $y$-axis projections are disjoint, and a set of rectangles is said to be independent if they are pairwise independent.  An {\it increasing set\/} of rectangles is an independent set of rectangles $\{R_1,\dots,R_m\}$ such that $R_2$ lies above and to the right of $R_1$, $R_3$ lies above and to the right of $R_2$, and so on.  {\it Decreasing sets\/} of rectangles are defined analogously.  Note that independent sets of rectangles can (essentially) be viewed as permutations; thus they fall under the purview of the Erd\H{o}s-Szekeres Theorem, so every independent set of $(m-1)^2+1$ rectangles contains either an increasing or a decreasing subset of $m$ rectangles.

Returning to the context of Theorem~\ref{gridding-characterization}, if the class $\C$ satisfies the hypotheses of that theorem then for any permutation $\pi\in\C$ the set $\R=\{\mbox{axis-parallel rectangles $R$} : \pi(R)\not\in\D\}$ does not contain an increasing or a decreasing set of $m$ rectangles (if it did, then each such rectangle would contain a basis element of $\D$, and thus $\pi$, and therefore $\C$, would contain $\beta_1\oplus\cdots\oplus\beta_m$ or $\beta_1\ominus\cdots\ominus\beta_m$ for some basis elements $\beta_1,\dots,\beta_m$ of $\D$).  The proof of the theorem is therefore completed with the following lemma.  We give a short proof of an exponential bound for this bound, a result which can also be derived from the work of Gy{\'a}rf{\'a}s and Lehel~\cite{gyarfas:a-helly-type-pr:}.  K{\'a}rolyi and Tardos~\cite{karolyi:on-point-covers:} take a different approach, which gives a polynomial bound.

\begin{lemma}\label{rectangles-lemma}
There is a function $f(m)$ such that for any collection $\R$ of axis-parallel rectangles in the plane which has no independent set of size $m$ or greater, there exists a set of $f(m)$ horizontal and vertical lines that slice every rectangle in $\R$.
\end{lemma}
\begin{proof}
Our proof is by induction on $m$; note that the base case $m=0$ is trivial.  We denote by $\proj_x R$ and $\proj_y R$ the projections of $R$ onto the $x$- and $y$-axes, respectively.  From these projections we define two structures on the set $\R$, a quasi-order $\subseteq_x$ and a graph $\sim_y$:
\begin{center}
\begin{tabular}{lcl}
$R\subseteq_x S$&if&$\proj_x R\subseteq\proj_x S$,\\
$R\sim_y S$&if&$\proj_y R\cap \proj_y S\neq\emptyset$.
\end{tabular}
\end{center}
Figure~\ref{fig-gridding-characterization} shows an example of these two structures.

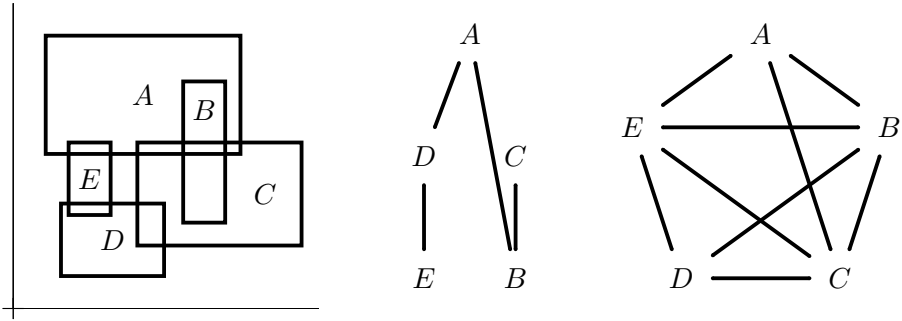
\begin{figure}
\begin{center}
\begin{tabular}{ccccc}
\psset{xunit=0.008in, yunit=0.008in}
\psset{linewidth=0.005in}
\begin{pspicture}(0,0)(200,200)
\psaxes[dy=1000,dx=1000](0,0)(200,200)
\psframe[linewidth=0.02in](30,20)(100,70)
\rput[c](65,45){$D$}
\psframe[linewidth=0.02in](20,100)(150,180)
\rput[c](85,140){$A$}
\psframe[linewidth=0.02in](35,60)(65,110)
\rput[c](50,85){$E$}
\psframe[linewidth=0.02in](80,40)(190,110)
\rput[c](165,75){$C$}
\psframe[linewidth=0.02in](110,55)(140,150)
\rput[c](125,130){$B$}
\end{pspicture}
&
\rule{10pt}{0pt}
&
\psset{xunit=0.008in, yunit=0.008in}
\psset{nodesep=0.1in}
\psset{linewidth=0.02in}
\begin{pspicture}(0,0)(80,200)
\rput(10,20){\rnode{E}{$E$}}
\rput(70,20){\rnode{B}{$B$}}
\rput(10,100){\rnode{D}{$D$}}
\rput(70,100){\rnode{C}{$C$}}
\rput(40,180){\rnode{A}{$A$}}
\ncline{c-c}{E}{D}
\ncline{c-c}{D}{A}
\ncline{c-c}{B}{A}
\ncline{c-c}{B}{C}
\end{pspicture}
&
\rule{10pt}{0pt}
&
\psset{xunit=0.008in, yunit=0.008in}
\psset{nodesep=0.1in}
\psset{linewidth=0.02in}
\begin{pspicture}(0,0)(200,200)
\rput(91.55417530, 180){\rnode{A}{$A$}}
\rput(175.6711533, 118.8854382){\rnode{B}{$B$}}
\rput(143.5413267, 20){\rnode{C}{$C$}}
\rput(39.56702390, 20){\rnode{D}{$D$}}
\rput(7.43719734, 118.8854382){\rnode{E}{$E$}}
\ncline{c-c}{A}{B}
\ncline{c-c}{A}{C}
\ncline{c-c}{A}{E}
\ncline{c-c}{B}{C}
\ncline{c-c}{B}{D}
\ncline{c-c}{B}{E}
\ncline{c-c}{C}{D}
\ncline{c-c}{C}{E}
\ncline{c-c}{D}{E}
\end{pspicture}
\end{tabular}
\end{center}
\caption{A collection of axis-parallel rectangles (left) with their $\subseteq_x$ quasi-order (center, which in this case happens to be a partial order) and $\sim_y$ graph (right).}\label{fig-gridding-characterization}.
\end{figure}

Consider first a clique, say $\K\subseteq\R$, in the $\sim_y$ graph.  Every pair of rectangles in this clique have intersecting $y$-projections, so $\bigcap_{R\in\K}\proj_y R\neq\emptyset$ (this is the one-dimensional version of Helly's Theorem).  Hence every clique in the $\sim_y$ graph can be sliced by a single horizontal line; in particular, we are done in the case where $\sim_y$ is complete.

\begin{figure}
\begin{center}
\begin{tabular}{ccccc}
\psset{xunit=0.004in, yunit=0.004in}
\psset{linewidth=0.005in}
\begin{pspicture}(0,0)(200,200)
\psframe[linecolor=white,fillstyle=solid,fillcolor=lightgray](0,0)(20,100)
\psline[linecolor=darkgray,linestyle=solid,linewidth=0.02in]{c-c}(0,100)(20,100)
\psline[linecolor=darkgray,linestyle=solid,linewidth=0.02in]{c-c}(20,0)(20,100)
\psaxes[dy=1000,dx=1000](0,0)(200,200)
\psline[linewidth=0.02in]{c-c}(30,20)(30,70)
\psline[linewidth=0.02in]{c-c}(30,70)(100,70)
\psline[linewidth=0.02in]{c-c}(100,70)(100,20)
\psline[linewidth=0.02in]{c-c}(100,20)(30,20)
\rput[c](65,45){$R_1$}
\psline[linewidth=0.02in]{c-c}(20,100)(20,180)
\psline[linewidth=0.02in]{c-c}(20,180)(150,180)
\psline[linewidth=0.02in]{c-c}(150,180)(150,100)
\psline[linewidth=0.02in]{c-c}(150,100)(20,100)
\rput[c](85,140){$R_2$}
\end{pspicture}
&
\rule{10pt}{0pt}
&
\psset{xunit=0.004in, yunit=0.004in}
\psset{linewidth=0.005in}
\begin{pspicture}(0,0)(200,200)
\psframe[linecolor=white,fillstyle=solid,fillcolor=lightgray](20,0)(150,100)
\psline[linecolor=darkgray,linestyle=solid,linewidth=0.02in]{c-c}(20,0)(20,100)
\psline[linecolor=darkgray,linestyle=solid,linewidth=0.02in]{c-c}(150,0)(150,100)
\psaxes[dy=1000,dx=1000](0,0)(200,200)
\psline[linewidth=0.02in]{c-c}(30,20)(30,70)
\psline[linewidth=0.02in]{c-c}(30,70)(100,70)
\psline[linewidth=0.02in]{c-c}(100,70)(100,20)
\psline[linewidth=0.02in]{c-c}(100,20)(30,20)
\rput[c](65,45){$R_1$}
\psline[linewidth=0.02in]{c-c}(20,100)(20,180)
\psline[linewidth=0.02in]{c-c}(20,180)(150,180)
\psline[linewidth=0.02in]{c-c}(150,180)(150,100)
\psline[linewidth=0.02in]{c-c}(150,100)(20,100)
\rput[c](85,140){$R_2$}
\end{pspicture}
&
\rule{10pt}{0pt}
&
\psset{xunit=0.004in, yunit=0.004in}
\psset{linewidth=0.005in}
\begin{pspicture}(0,0)(200,200)
\psframe[linecolor=white,fillstyle=solid,fillcolor=lightgray](150,100)(200,0)
\psline[linecolor=darkgray,linestyle=solid,linewidth=0.02in]{c-c}(150,100)(200,100)
\psline[linecolor=darkgray,linestyle=solid,linewidth=0.02in]{c-c}(150,0)(150,100)
\psaxes[dy=1000,dx=1000](0,0)(200,200)
\psline[linewidth=0.02in]{c-c}(30,20)(30,70)
\psline[linewidth=0.02in]{c-c}(30,70)(100,70)
\psline[linewidth=0.02in]{c-c}(100,70)(100,20)
\psline[linewidth=0.02in]{c-c}(100,20)(30,20)
\rput[c](65,45){$R_1$}
\psline[linewidth=0.02in]{c-c}(20,100)(20,180)
\psline[linewidth=0.02in]{c-c}(20,180)(150,180)
\psline[linewidth=0.02in]{c-c}(150,180)(150,100)
\psline[linewidth=0.02in]{c-c}(150,100)(20,100)
\rput[c](85,140){$R_2$}
\end{pspicture}
\\
(a)&&(b)&&(c)
\\[20pt]
\psset{xunit=0.004in, yunit=0.004in}
\psset{linewidth=0.005in}
\begin{pspicture}(0,0)(200,200)
\psframe[linecolor=white,fillstyle=solid,fillcolor=lightgray](30,70)(0,200)
\psline[linecolor=darkgray,linestyle=solid,linewidth=0.02in]{c-c}(30,70)(0,70)
\psline[linecolor=darkgray,linestyle=solid,linewidth=0.02in]{c-c}(30,70)(30,200)
\psaxes[dy=1000,dx=1000](0,0)(200,200)
\psline[linewidth=0.02in]{c-c}(30,20)(30,70)
\psline[linewidth=0.02in]{c-c}(30,70)(100,70)
\psline[linewidth=0.02in]{c-c}(100,70)(100,20)
\psline[linewidth=0.02in]{c-c}(100,20)(30,20)
\rput[c](65,45){$R_1$}
\psline[linewidth=0.02in]{c-c}(20,100)(20,180)
\psline[linewidth=0.02in]{c-c}(20,180)(150,180)
\psline[linewidth=0.02in]{c-c}(150,180)(150,100)
\psline[linewidth=0.02in]{c-c}(150,100)(20,100)
\rput[c](85,140){$R_2$}
\end{pspicture}
&
\rule{10pt}{0pt}
&
\psset{xunit=0.004in, yunit=0.004in}
\psset{linewidth=0.005in}
\begin{pspicture}(0,0)(200,200)
\psframe[linecolor=white,fillstyle=solid,fillcolor=lightgray](30,70)(100,200)
\psline[linecolor=darkgray,linestyle=solid,linewidth=0.02in]{c-c}(30,70)(30,200)
\psline[linecolor=darkgray,linestyle=solid,linewidth=0.02in]{c-c}(100,70)(100,200)
\psline[linecolor=darkgray,linestyle=solid,linewidth=0.02in]{c-c}(30,70)(100,70)
\psaxes[dy=1000,dx=1000](0,0)(200,200)
\psline[linewidth=0.02in]{c-c}(30,20)(30,70)
\psline[linewidth=0.02in]{c-c}(30,70)(100,70)
\psline[linewidth=0.02in]{c-c}(100,70)(100,20)
\psline[linewidth=0.02in]{c-c}(100,20)(30,20)
\rput[c](65,45){$R_1$}
\psline[linewidth=0.02in]{c-c}(20,100)(20,180)
\psline[linewidth=0.02in]{c-c}(20,180)(150,180)
\psline[linewidth=0.02in]{c-c}(150,180)(150,100)
\psline[linewidth=0.02in]{c-c}(150,100)(20,100)
\rput[c](85,140){$R_2$}
\end{pspicture}
&
\rule{10pt}{0pt}
&
\psset{xunit=0.004in, yunit=0.004in}
\psset{linewidth=0.005in}
\begin{pspicture}(0,0)(200,200)
\psframe[linecolor=white,fillstyle=solid,fillcolor=lightgray](100,70)(200,200)
\psline[linecolor=darkgray,linestyle=solid,linewidth=0.02in]{c-c}(100,70)(200,70)
\psline[linecolor=darkgray,linestyle=solid,linewidth=0.02in]{c-c}(100,70)(100,200)
\psaxes[dy=1000,dx=1000](0,0)(200,200)
\psline[linewidth=0.02in]{c-c}(30,20)(30,70)
\psline[linewidth=0.02in]{c-c}(30,70)(100,70)
\psline[linewidth=0.02in]{c-c}(100,70)(100,20)
\psline[linewidth=0.02in]{c-c}(100,20)(30,20)
\rput[c](65,45){$R_1$}
\psline[linewidth=0.02in]{c-c}(20,100)(20,180)
\psline[linewidth=0.02in]{c-c}(20,180)(150,180)
\psline[linewidth=0.02in]{c-c}(150,180)(150,100)
\psline[linewidth=0.02in]{c-c}(150,100)(20,100)
\rput[c](85,140){$R_2$}
\end{pspicture}
\\
(d)&&(e)&&(f)
\end{tabular}
\end{center}
\caption{The six regions from the proof of Lemma~\ref{rectangles-lemma}.}\label{fig-rectangles-lemma}
\end{figure}
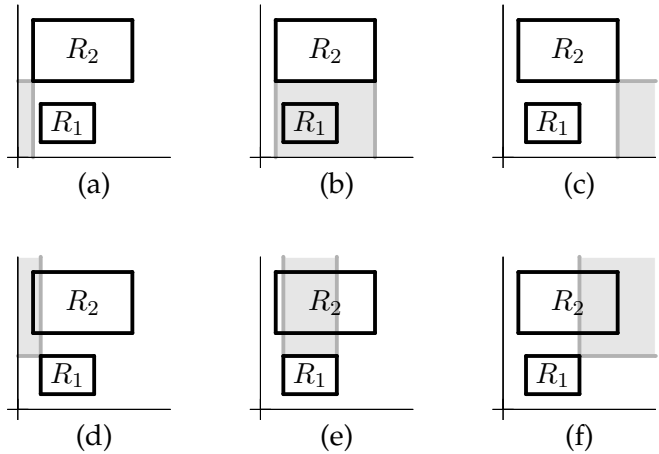

We therefore assume that $\sim_y$ is not complete and define the {\it height\/} of the rectangle $R$ in the $\subseteq_x$ quasi-order, which we denote by $\height_x R$, as one less than the greatest number of distinct elements in a nondecreasing (with respect to $\subseteq_x$) chain which ends at $R$.  For example, in Figure~\ref{fig-rectangles-lemma}, $\height_x E=\height_x B=0$, $\height_x D=\height_x C=1$, and $\height_x A=2$.  Choose $R_1\not\sim_y R_2$ from $\R$ to minimize $\max(\height_x(R_1),\height_x(R_2))$, and without loss of generality suppose that $\height_x(R_1)\le\height_x(R_2)$ and that $R_1$ lies below $R_2$.  Therefore $\{R\in\R : \height_x(R)<\height_x(R_2)\}$ forms a clique in the $\sim_y$ graph, and thus by the above, these rectangles can be sliced with a single horizontal line.

Now consider the regions shown in Figure~\ref{fig-rectangles-lemma}.  By induction and our choice of $R_1$ and $R_2$, we have the following:
\begin{enumerate}
\item[(a)] The rectangles completely contained in this region, which consists of all points strictly below and to the left of $R_2$, cannot contain an independent set of $m-1$ or more rectangles, as otherwise $R_2$ together with this independent set would form an independent set of size $m$.  Thus this region can be sliced by $f(m-1)$ horizontal and vertical lines by induction.
\item[(b)] This region consists of all points strictly below $R_2$ whose $x$-coordinate lies strictly between the left and right edges of $R_2$.  Therefore if the rectangle $R$ lies in this region then $\height_x(R)<\height_x(R_2)$.  Thus by our choice of $R_2$ these rectangles form a clique in the $\sim_y$ graph and so can be sliced by a single horizontal line.
\item[(c)] As with (a), these rectangles can be sliced by $f(m-1)$ horizontal and vertical lines.
\item[(d)] As with (a), these rectangles can be sliced by $f(m-1)$ horizontal and vertical lines.
\item[(e)] As with (b), if the rectangle $R$ is completely contained in this region, then $\height_x(R)<\height_x(R_1)\le\height_x(R_2)$, and thus by our choice of $R_2$ these rectangles for a clique in the $\sim_y$ graph and so can be sliced by a single horizontal line.
\item[(f)] As with (a), these rectangles can be sliced by $f(m-1)$ horizontal and vertical lines.
\end{enumerate}

Thus we have found a collection of $4f(m-1)+2$ vertical and horizontal lines which slice every rectangle properly contained in one of the regions (a)--(f).  Furthermore, the $8$ lines which coincide with the sides of $R_1$ and $R_2$ slice every rectangle that is not properly contained in one of these regions, and so we may take $f(m)=4f(m-1)+10$, completing the proof.
\end{proof}

A special case of Theorem~\ref{gridding-characterization}, the characterization of monotone griddable classes, appeared in Huczynska and Vatter~\cite{huczynska:grid-classes-an:}.  Note that the condition given there --- a permutation class is monotone griddable if and only if it does not contain arbitrarily long sums of $21$ or skew sums of $12$ --- is merely a simplification of the conditions given by Theorem~\ref{gridding-characterization}%
\footnote{{\it Proof.}\quad Take $\D$ to contain the monotone permutations, i.e., $\D=\Av(12)\cup\Av(21)$, and note that a class is $\D$-griddable if and only if it is monotone griddable by Proposition~\ref{grid-union}.  The basis of $\D$ is $\{132,213,231,312\}$, so Theorem~\ref{gridding-characterization} implies that the permutation class $\C$ is $\D$-griddable if and only if it does not contain arbitrarily long sums or skew sums of any of these elements, a condition equivalent to not containing arbitrarily long sums of $21$ or skew sums of $12$.\qed}.

Moving beyond griddability, we say that the class $\D$ is {\it grid irreducible\/} if $\D$ is not $\E$-griddable for any proper subclass $\E\subsetneq\D$.  Theorem~\ref{gridding-characterization} gives us, almost immediately, a characterization of the grid irreducible classes.

\begin{proposition}\label{grid-irreducible}
The permutation class $\D$ is grid irreducible if and only if $\D=\{1\}$ or for every $\pi\in\D$ either $\bigoplus\pi$ or $\bigominus\pi$ is contained in $\D$.
\end{proposition}
\begin{proof}
If $\D$ is finite then it is clear that $\D$ is grid irreducible if and only if $\D=\{1\}$, so we suppose that $\D$ is infinite.  If there is some $\pi\in\D$ such that neither $\bigoplus\pi$ nor $\bigominus\pi$ is contained in $\D$ then Theorem~\ref{gridding-characterization} shows that $\D$ is $\D\cap\Av(\pi)$-griddable, and thus $\D$ is not grid irreducible.  If, on the other hand, $\D$ contains arbitrarily long direct sums or skew sums of each of its members then it is clear that $\D$ is not $\E$-griddable for any proper subclass $\E\subsetneq\D$.
\end{proof}

Note that Proposition~\ref{grid-irreducible} does not guarantee that every class $\C$ is $\D$-griddable for a grid irreducible subclass $\D\subseteq\C$, and indeed, this does not necessarily hold%
\footnote{For example, take $A=\{a_1,a_2,\dots\}$ to be an infinite antichain and let $\C$ denote the closure of the set of permutations of the form $a_{i_1}\oplus a_{i_2}\oplus\cdots$ where $1\le i_1<i_2<\cdots$.  Suppose to the contrary that $\C$ is $\D$-griddable for a grid irreducible subclass $\D\subseteq\C$, which then, by Proposition~\ref{grid-irreducible}, cannot contain any member of $A$.  This, however, implies that each $a_i$ contains a basis element of $\D$, and thus $\C$ contains arbitrarily long direct sums of basis elements of $\D$ and so is not $\D$-griddable by Theorem~\ref{gridding-characterization}.
}%
; however, it is true in an important special case.

\begin{proposition}\label{pwo-grid-irreducible}
If the permutation class $\C$ is pwo then $\C$ is $\D$-griddable for the grid irreducible subclass $\D=\{\pi\in\C: \mbox{either $\bigoplus\pi$ or $\bigominus\pi$ is contained in $\C$}\}$.
\end{proposition}
\begin{proof}
Choose $\pi_1\in\C$ such that $\C$ contains neither $\bigoplus\pi_1$ nor $\bigominus\pi_1$.  It follows from Theorem~\ref{gridding-characterization} that $\C$ is $\C\setminus\{\pi_1\}$-griddable.  If
$$
\C\setminus\{\pi_1\}=\D=\{\pi\in\C: \mbox{either $\bigoplus\pi$ or $\bigominus\pi$ is contained in $\C$}\}
$$
then we are done, so we suppose that there is some $\pi_2\in\C\setminus\{\pi_1\}$ such that neither $\bigoplus\pi_2$ nor $\bigominus\pi_2$ is contained in $\C$.  Again we have from Theorem~\ref{gridding-characterization} that $\C$ is $\C\setminus\{\pi_2\}$-griddable, and thus by Proposition~\ref{grid-intersection}, $\C$ is $\C\setminus\{\pi_1,\pi_2\}$-griddable.  We are done if $\C\setminus\{\pi_1,\pi_2\}=\D$, and in the other case we may choose $\pi_3\in\C\setminus\{\pi_1,\pi_2\}$ such that neither $\bigoplus\pi_3$ nor $\bigominus\pi_3$ is contained in $\C$.  Continuing in this manner we construct a descending chain
$$
\C
\supsetneq\C\setminus\{\pi_1\}
\supsetneq\C\setminus\{\pi_1,\pi_2\}
\supsetneq\C\setminus\{\pi_1,\pi_2,\pi_3\}
\supsetneq\cdots
\supseteq\D
$$
of classes such that for all $i$, $\C$ is $\C\setminus\{\pi_1,\dots,\pi_i\}$-griddable.  Because $\C$ is pwo, it satisfies the descending chain condition by Proposition~\ref{pwo-subclasses-dcc}, and so this chain must terminate, say at $\C\setminus\{\pi_1,\dots,\pi_m\}$.  This means that for every $\pi\in\C\setminus\{\pi_1,\dots,\pi_m\}$, $\C$ contains either $\bigoplus\pi$ or $\bigominus\pi$.  Therefore $\C\setminus\{\pi_1,\dots,\pi_m\}=\D$ and $\C$ is $\D$-griddable, as desired.
\end{proof}

We conclude this section with the following result; the inference important to us is that the combination of atomicity and grid irreducibility is quite strong, which is needed for the proof of Theorem~\ref{gr-atomic-grid-irreducible}.

\begin{proposition}\label{atomic-grid-irreducible}
The following conditions on a permutation class $\C$ are equivalent:
\begin{enumerate}
\item[(1)] $\C$ is both atomic and grid irreducible,
\item[(2)] for every $\pi,\sigma\in\C$, either $\pi\oplus\sigma\in\C$ or $\pi\ominus\sigma\in\C$, and
\item[(3)] $\C$ is either sum or skew sum complete.
\end{enumerate}
\end{proposition}
\begin{proof}
If a permutation class is either sum or skew sum complete then it satisfies the joint embedding property, so is atomic by Theorem~\ref{atomic-tfae}, and also satisfies the conditions of Proposition~\ref{grid-irreducible}, so is grid irreducible, verifying that (3) implies (1).

Suppose now that the class $\C$ is both atomic and grid irreducible.  For any two permutations $\pi,\sigma\in\C$, $\C$ contains a permutation $\tau\ge\pi,\sigma$ because it is atomic.  Then because $\C$ is grid irreducible, it contains either $\tau\oplus\tau$ or $\tau\ominus\tau$, and thus either $\pi\oplus\sigma\in\C$ or $\pi\ominus\sigma\in\C$, so (1) implies (2).

To show that (2) implies (3) and complete the proof, list the members of $\C$ as $\pi_1,\pi_2,\dots$ (which can be done because every permutation class is countable) and for each $k$ choose a permutation $\sigma_k\in\C$ which contains $\pi_1,\dots,\pi_k$ (that $\sigma_k$ exists is guaranteed by (2)).  Now set $\tau_1=\sigma_1$ and for $k\ge 1$, take $\tau_{k+1}\in\C$ to be either $\tau_k\oplus\sigma_{k+1}$ or $\tau_k\ominus\sigma_{k+1}$, depending on which permutation lies in $\C$.  Either infinitely many of the $\tau_k$s are formed by sums or infinitely many are formed by skew sums (or both); in the former case $\C$ is sum complete, in the latter, skew sum complete.
\end{proof}

\section{Griddings of Small Permutation Classes}\label{sec-subst-gridding}

Having established the basic properties of grid classes, we move on to what can be said about small permutation classes.  Given a small permutation class $\C$, we would like to find a ``nice'' class $\D$ for which $\C$ is $\D$-griddable.  For our purposes, ``nice'' means that $\D$ should have only finitely many simple permutations and bounded substitution depth, a concept defined within this section which is required in Section~\ref{sec-gridding-condition}.

Let $\O$ denote the closure of the set of oscillations (that is, $\O$ consists of all oscillations and every permutation contained in an oscillation) and $\O_k$ denote the closure of the set of oscillations of length at most $k$.  For the moment we simply consider increasing oscillations, although $\O$ also contains decreasing oscillations.  The set of increasing oscillations forms two chains under the containment order, and  each increasing oscillation of length $k\ge 4$ contains both increasing oscillations of length $k-1$ (see Figure~\ref{fig-inc-osc}).  It is easy to verify (e.g., using Proposition~\ref{sum-indecomp-connected}) that if $\pi$ is an increasing oscillation of length $k\ge 3$ then the nonempty sum indecomposable permutations (i.e., the increasing oscillations) contained in $\pi$ have the generating function
$$
x+x^2+2x^3+\cdots+2x^{k-1}+x^k
=
\frac{x+x^3-x^k-x^{k+1}}{1-x},
$$
so $\bigoplus\pi$ has the generating function
$$
\frac{1-x}{1-2x-x^3+x^k+x^{k+1}},
$$
from which it follows, using Pringsheim's Theorem, that $\gr(\bigoplus\pi)<\kappa$.  Therefore, since we want to find a class $\D$ for which every class with lower growth rate less than $\kappa$ has a $\D$-gridding, $\D$ must contain all increasing oscillations, and by symmetry, all decreasing oscillations, i.e., we must have $\O\subseteq\D$.  This, however, is not enough: $1432$ does not lie in $\O$, yet $\bigoplus 1432$ has the generating function $1/(1-x-x^2-x^3)$ and thus the growth rate $1.83928\dots$, much less than our goal.

Instead, we consider the substitution completion $\S(\O)$.  A routine calculation of the basis of $\S(\O)$ (see Propositions~\ref{prop-basis-WO} and \ref{prop-WO-gridding}) that shows every small permutation class is $\S(\O)$-griddable.  The following proposition then follows.

\begin{proposition}\label{prop-WOk-gridding}
Every permutation class with lower growth rate less than $\kappa$ is $\S(\O_k)$-griddable for some $k$.
\end{proposition}
\begin{proof}
If $\C$ is a permutation class with growth rate less than $\kappa$ then Proposition~\ref{prop-WO-gridding} shows that $\C$ is $\S(\O)$-griddable, and thus $\C$ is $\C\cap\S(\O)$-griddable.  If $\C$ were to contain every increasing oscillation or every decreasing oscillation, then Proposition~\ref{prop-growth-osc} would imply that $\lgr(\C)\ge\kappa$, a contradiction.  Thus, there must be some $k$ for which $\C$ has an $\S(\O_k)$-gridding.
\end{proof}

We move on to bounded substitution depth, for which we much first define the substitution decomposition tree of a permutation.  Proposition~\ref{simple-decomp-1} shows that every permutation is the inflation of a unique simple permutation and, moreover, that the intervals in such an inflation are unique unless the permutation is the inflation of either $12$ or $21$, i.e., unless it is sum or skew decomposable.  There are two ways that this non-uniqueness is typically handled.  The first way, typically used for exact enumeration, is to write $\pi$ as $\alpha_1\oplus\alpha_2$ where $\alpha_1$ is sum indecomposable and $\alpha_2$ is arbitrary.  For our purposes it is more natural to write $\pi$ as $\alpha_1\oplus\cdots\oplus\alpha_m$ where each $\alpha_i$ is sum indecomposable.  By recursively decomposing each interval $\alpha_i$ in this manner, we obtain the {\it substitution decomposition tree\/} of $\pi$, which is formed by decomposing $\pi$ as one of
\begin{itemize}
\item $\sigma[\alpha_1,\dots,\alpha_m]$ where $\sigma$ is a nonmonotone simple permutation,
\item $\alpha_1\oplus\cdots\oplus\alpha_m$ where each $\alpha_i$ is sum indecomposable, or
\item $\alpha_1\ominus\cdots\ominus\alpha_m$ where each $\alpha_i$ is skew sum indecomposable.
\end{itemize}
See Figure~\ref{fig-tree-479832156} for an example; note that in this example, we decompose $321$ as $1\ominus 1\ominus 1$, whereas the alternative would be to decompose it as $1\ominus 21$ and then decompose $21$ again.  The substitution decomposition tree of a permutation is a rooted tree, and so we use terms such as child, parent, ancestor, and descendent when discussing it.  The {\it substitution depth\/} of $\pi$ is then the height of its substitution decomposition tree, so for example, the substitution depth of the permutation from Figure~\ref{fig-tree-479832156} is $3$, while the substitution depth of any nontrivial simple or monotone permutation is $1$.

\begin{figure}
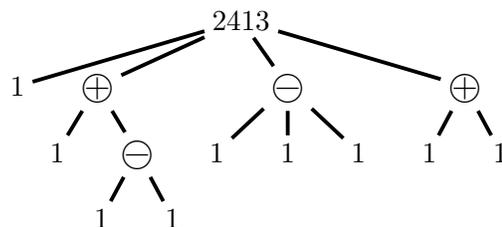

$$
\pstree[nodesep=3pt,levelsep=25pt,linewidth=0.02in]{\TR{2413} }{ 
\TR{1} 
\pstree{ \TR{\bigoplus} }{ 
\TR{1} 
\pstree{ \TR{\bigominus} }{ 
\TR{1} 
\TR{1} 
}
}
\pstree{ \TR{\bigominus} }{
\TR{1} 
\TR{1} 
\TR{1}
} 
\pstree{ \TR{\bigoplus} }{ 
\TR{1} 
\TR{1} 
}
} 
$$
\caption{The substitution decomposition tree of $479832156$.}\label{fig-tree-479832156}
\end{figure}

\begin{proposition}\label{prop-long-path-wedge}
If the permutation $\pi$ has substitution depth at least $8n$ then $\pi$ contains a wedge alternation of length at least $n$.
\end{proposition}
\begin{proof}
Suppose that the substitution depth of the permutation $\pi$ is at least $8n$.  By definition, this means that the substitution decomposition tree of $\pi$ contains a path of length $8n$.  We first claim that we can find a path of length $4n$ which alternates between nodes of type $\oplus$ and $\ominus$.  Note that every node in the substitution decomposition tree is labeled by either $\oplus$, $\ominus$, or a nonmonotone simple permutation.  Every entry of a nonmonotone simple permutation participates in copies of both $12$ and a $21$ (as otherwise the permutation would be sum or skew sum decomposable), so such a node can be replaced (by considering a subpermutation of $\pi$) by either $\oplus$ or $\ominus$.  Therefore, since $\oplus$ and $\ominus$ nodes cannot repeat, and nodes labeled by nonmonotone simple permutations can function as either type, we can find an alternating path of length $4n$, as desired.


With an alternating path of length $4n$ in hand, we divide the $\oplus$ nodes by whether the path lies to the left or to the right.  By considering the reverse of $\pi$ if necessary, we may assume that in at least half of the nodes, the path lies on the left of the tree.  Therefore, we can find a path of length $2n$, alternating between nodes labeled $\oplus$ and $\ominus$, in which the path lies on the left of all $\oplus$ nodes.  We can then, by considering the inverse of $\pi$ if necessary, find a path of length $n$ which alternates between $\oplus$ and $\ominus$ in which the path lies on the left of all nodes.  From this path, it follows that $\pi$ contains a permutation of length $n$ of the form $((((\cdots)\ominus 1)\oplus 1)\ominus 1)\oplus 1$ or $((((\cdots) \oplus 1) \ominus 1) \oplus 1) \ominus 1$, both of which are wedge alternations.
\end{proof}

Since wedge alternations have unbounded substitution depth and a growth rate of only $2$, small classes can have unbounded substitution depth.  However, as we show next, small classes can always be gridded by classes with bounded substitution depth.

\begin{theorem}\label{thm-bounded-subst-depth}
If the permutation class $\C$ does not contain arbitrarily long sums or skew sums of arbitrarily long wedge alternations (which by Proposition~\ref{prop-sum-wedge} would imply that $\lgr(\C)\ge 2.61803$), then $\C$ is $\D$-griddable for a class $\D$ with bounded substitution depth.
\end{theorem}
\begin{proof}
Let $\D_k$ denote the class of all permutations of substitution depth at most $k$.  If $\C$ cannot be $\D_k$-gridded for any $k$, then Theorem~\ref{gridding-characterization} shows that $\C$ contains arbitrarily long sums or skew sums of basis elements of $\D_k$ for integers $k$.  Proposition~\ref{prop-long-path-wedge} shows that all basis elements of $\D_{8k-1}$ contain wedge alternations of length at least $k$, implying that $\C$ contains arbitrarily long sums or skew sums of arbitrarily long wedge alternations, as desired.
\end{proof}


We conclude the section by collecting what we now know about griddings of small classes.

\begin{theorem}\label{thm-subst-gridding-main}
If the permutation class $\C$ satisfies $\lgr(\C)<\kappa$ then $\C$ is $\D$-griddable for a class $\D$ with finitely many simple permutations and bounded substitution depth.
\end{theorem}
\begin{proof}
Proposition~\ref{prop-WOk-gridding} shows that $\C$ is $\D$-griddable for a class $\D$ with finitely many simple permutations, while Theorem~\ref{thm-bounded-subst-depth} shows that $\C$ is $\E$-griddable for a class $\E$ with bounded substitution depth.  Therefore $\C$ is $\D\cap\E$-griddable by Proposition~\ref{grid-intersection}.
\end{proof}

Classes with only finitely many simple permutations and bounded substitution depth are especially nice since they are pwo (by Proposition~\ref{fin-simples-pwo}) and have rational generating functions\footnote{While not explicitly mentioned in Albert and Atkinson~\cite{albert:simple-permutat:} or Brignall, Huczynska, and Vatter~\cite{brignall:simple-permutat:}, this result follows readily from either approach.}.  Our primary interest in such classes stems from how nicely they can be ``sliced'', as we discuss in the next section.

\section{Refined Griddings of Small Permutation Classes}\label{sec-gridding-condition}

Our goal in this section, which constitutes the most technical component of the proof, is to restrict the possible gridding matrices of small classes.  To motivate the proof of that result, we first consider a simpler situation.  In their characterization of classes with polynomial growth, Huczynska and Vatter proved the following.

\begin{theorem}[Huczynska and Vatter~\cite{huczynska:grid-classes-an:}]\label{alternations-gridding-mono}
If the permutation class $\C$ is monotone griddable and does not contain arbitrarily long alternations, then $\C$ is $\M$-griddable for a matrix $\M$ whose cell graph is edgeless and in which every vertex is labeled by a monotone class.
\end{theorem}

We begin by proving the following generalization.

\begin{theorem}\label{alternations-gridding}
Suppose that the permutation class $\C$ is $\D$-griddable, that $\D$ has only finitely many simple permutations and bounded substitution depth, and that $\C$ does not contain arbitrarily long alternations.  Then $\C$ is $\M$-griddable for a matrix $\M$ whose cell graph is edgeless and in which every vertex is labeled by $\D$.
\end{theorem}

Note that this is indeed a generalization of Theorem~\ref{alternations-gridding-mono}, since the class of monotone permutations contains only three simple permutations ($1$, $12$, and $21$) and every nontrivial monotone permutation has substitution depth $1$.

We sketch the argument we use to prove Theorem~\ref{alternations-gridding} first, in order to draw attention to the difficulties involved in this generalization (which also arise in the proof of the major result of this section, Theorem~\ref{small-classes-gridding}).  After discussing how to deal with these issues, we present a formal proof.

Choose a $t\times u$ matrix $\N$ whose entries are all subclasses of $\D$ (or, without loss of generality, equal to $\D$ itself) such that $\C$ is $\N$-griddable and does not contain alternations of length $2m$.  Now choose an $\N$-gridded permutation $\pi\in\C$.  We seek to break up the small alternations in $\pi$ with a refined gridding which does not use too many additional lines.

We say that a rectangle $R$ is {\it separated\/} if it is axis-parallel, $\pi(R)$ is completely contained in one cell of the chosen gridding of $\pi$, and $\pi(R)$ contains (at least) two entries which are separated by an entry in a different cell.  Let $\R$ denote the collection of all separated rectangles in this gridding of $\pi$; it can be shown from the bound on the length of alternations in $\C$ that $\R$ contains no independent set of size $4mtu$, and therefore, by Lemma~\ref{rectangles-lemma}, $\R$ can be sliced by a set of at most $f(4mtu)$ vertical and horizontal lines.  Denote this set by $\L$.

\begin{figure}
\begin{center}
\begin{tabular}{ccc}
\psset{xunit=0.01in, yunit=0.01in}
\psset{linewidth=0.005in}
\begin{pspicture}(0,0)(110,110)
\psaxes[dy=10,Dy=1,dx=10,Dx=1,tickstyle=bottom,showorigin=false,labels=none](0,0)(110,110)
\psline[linecolor=darkgray,linestyle=solid,linewidth=0.02in](105,0)(105,115)
\pscircle*(10,40){0.04in}
\pscircle*(20,10){0.04in}
\pscircle*(30,60){0.04in}
\pscircle*(40,30){0.04in}
\pscircle*(50,80){0.04in}
\pscircle*(60,50){0.04in}
\pscircle*(70,100){0.04in}
\pscircle*(80,70){0.04in}
\pscircle*(90,110){0.04in}
\pscircle*(110,20){0.04in}
\pscircle*(100,90){0.04in}
\end{pspicture}
&\rule{10pt}{0pt}&
\psset{xunit=0.01in, yunit=0.01in}
\psset{linewidth=0.005in}
\begin{pspicture}(0,0)(110,110)
\psaxes[dy=10,Dy=1,dx=10,Dx=1,tickstyle=bottom,showorigin=false,labels=none](0,0)(110,110)
\psline[linecolor=darkgray,linestyle=solid,linewidth=0.02in](0,15)(115,15)
\pscircle*(10,70){0.04in}
\pscircle*(20,10){0.04in}
\pscircle*(30,60){0.04in}
\pscircle*(40,80){0.04in}
\pscircle*(50,50){0.04in}
\pscircle*(60,90){0.04in}
\pscircle*(70,40){0.04in}
\pscircle*(80,100){0.04in}
\pscircle*(90,30){0.04in}
\pscircle*(100,110){0.04in}
\pscircle*(110,20){0.04in}
\end{pspicture}
\end{tabular}
\end{center}
\caption{On the left, an example of a gridded permutation formed by adding a single entry to an increasing oscillation sequence.  On the right, an example of a gridded permutation formed by adding a single entry to a wedge alternation.  Both examples demonstrate that we must be careful when generalizing Theorem~\ref{alternations-gridding-mono}.}
\label{fig-careful-alternations-gridding-mono}
\end{figure}
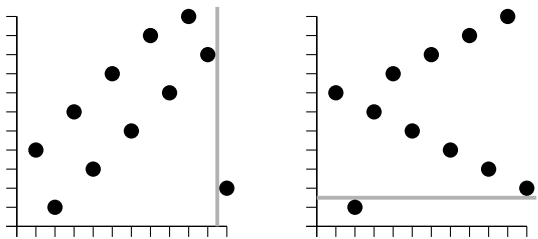

We now describe how these $f(4mtu)$ vertical and horizontal lines can be used to construct a refined gridding of $\pi$ with the properties demanded by Theorem~\ref{alternations-gridding}.  Figure~\ref{fig-careful-alternations-gridding-mono} shows that we need to be careful.  In the permutation on the left of this figure, all of the separated rectangles can be sliced by a single line.  However, if we consider the refined gridding of this permutation given by such a line, we encounter new separated rectangles.  Slicing those rectangles with a new line then leads to even more separated rectangles.  This is why we assume that $\D$ contains only finitely many simple permutations.  The permutation on the right of Figure~\ref{fig-careful-alternations-gridding-mono} exhibits similar behavior, showing why we assume that $\D$ has bounded substitution depth.

Our problem is therefore to use the lines $\L$ to refine the gridding of $\pi$ without introducing a large amount of new separated rectangles.  To do so, we focus on a single cell of the given gridding of $\pi$, which gives rise to a permutation $\tau\in\D$.  This permutation is sliced by some of the lines in $\L$.  We say that an interval of $\tau$ is {\it unsliced\/} (technically, with respect to $\L$) if no line in $\L$ separates two entries of the interval.  A interval of $\tau$ is then a {\it maximal unsliced interval\/} if it is as large as possible with this property.  As the maximal unsliced intervals partition the entries of $\tau$, we seek to bound their number.  Note that a simple permutation of length $n$ will have either $1$ or $n$ maximal unsliced intervals, as it will either be sliced (resulting in $n$ maximal unsliced intervals, all trivial) or not sliced (resulting in it itself being its lone maximal unsliced interval).  On the contrary, a monotone permutation sliced by $\ell$ lines will have at most $\ell+1$ maximal unsliced intervals.  Our next lemma shows how our conditions on the class $\D$ allow us to bound the number of maximal unsliced intervals.

\begin{lemma}\label{lem-max-unsliced-ints}
Suppose that the permutation class $\D$ has only finitely many simple permutations and bounded substitution depth.  Then there is a function $g(\ell)$ such that for any $\tau\in\D$ and any collection of $\ell$ vertical and horizontal lines slicing $\tau$, $\tau$ has at most $g(\ell)$ maximal unsliced intervals.
\end{lemma}
\begin{proof}
We prove the lemma by induction on $\ell$.  Noting that we may take $g(0)=1$, we now assume that $g(\ell-1)$ has been defined, and give a bound for $g(\ell)$.  Take $\tau\in\D$ and a set $\L$ of $\ell$ vertical or horizontal lines slicing $\tau$.  Now choose a set $\L'\subset\L$ with $\ell-1$ lines.  This collection defines a set of unsliced intervals $\mathcal{I}$ which we know from induction can contain at most $g(\ell-1)$ maximal elements.

Now consider the effect of the new line on the previously unsliced intervals $\mathcal{I}$ that it slices.  Suppose that the new line slices through an interval of $\mathcal{I}$ as otherwise there is nothing to prove.  Hence this line slices through a unique minimal interval of $\mathcal{I}$, and therefore the intervals sliced by the new line are precisely those containing this minimal interval.  In other words, the intervals sliced by the new line form a path in the substitution decomposition tree of $\tau$.

Given any previously unsliced interval sliced by this new line, there are two cases, depending on what type of node the interval is in the substitution decomposition tree of $\tau$.  If this interval decomposes as the inflation of a nonmonotone simple permutation, i.e., it is of the form $\sigma[\alpha_1,\dots,\alpha_m]$ for a simple permutation $\sigma$ with $m\ge 4$, then if the new line slices through $\alpha_i$, it creates $m-1$ new maximal intervals $\alpha_1,\dots,\alpha_{i-1},\alpha_{i+1},\dots,\alpha_m$, while if the line slices between two intervals, then each $\alpha_i$ becomes a new maximal unsliced interval.  Otherwise the interval decomposes as $\alpha_1\oplus\cdots\oplus\alpha_m$ or $\alpha_1\ominus\cdots\ominus\alpha_m$.  As both cases are similar, we consider the former.  If the line slices through $\alpha_i$, then $\alpha_1\oplus\cdots\oplus\alpha_{i-1}$ and $\alpha_{i+1}\oplus\cdots\oplus\alpha_m$ become maximal unsliced intervals, while if it slices between $\alpha_{i-1}$ and $\alpha_i$ then the new maximal unsliced intervals are $\alpha_1\oplus\cdots\oplus\alpha_{i-1}$ and $\alpha_{i}\oplus\cdots\oplus\alpha_m$.

Thus we have shown that the new line slices through at most one interval of each height in the substitution decomposition tree of $\tau$, and that at each height it creates at most $\max(s,3)$ new maximal intervals, where $s$ is the length of the longest simple permutation in $\D$.  Therefore, if we suppose that the substitution depth of $\D$ is $h$, we have shown that any $\ell$ vertical and horizontal lines create at most $g(\ell-1)+h\max(s,3)$ new maximal unsliced intervals, completing the proof.
\end{proof}

We will be in a position to complete this sketch of our proof after one more definition.  Given a set of points in the plane, we define their {\it rectangular hull\/} to be the smallest axis-parallel rectangle containing them.

Recall that we had an $\N$-gridded permutation $\pi$ and a collection of at most $f(4mtu)$ vertical and horizontal lines, $\L$, which slice through every separated rectangle of $\pi$.  Lemma~\ref{lem-max-unsliced-ints} shows that these lines create at most $g(f(4mtu))$ maximal unsliced intervals in each cell of the $\N$-gridding of $\pi$, and thus at most $tu\cdot g(f(4mtu))$ maximal unsliced intervals in total.  Let $\HH$ denote the rectangular hulls of these maximal unsliced intervals, and let $\L'$ denote the set of vertical and horizontal lines which coincide with the sides of these hulls.

Note that no two hulls from the same cell of the $\N$-gridding of $\pi$ may overlap, as the hulls are formed from intervals.  Furthermore, by their very definition, the hulls in $\HH$ have points on each of their sides (and also, because $\pi$ is a permutation, two distinct hulls cannot both have points on the same vertical or horizontal line).  Thus, were two of these hulls to overlap either horizontally or vertically, both would contain separated rectangles, a contradiction to our choice of $\L'$.  In other words, $\HH$ is an independent set.  Now let $\M$ denote the matrix which corresponds to the gridding of $\pi$ given by the lines $\L'$, i.e., the matrix in which $\M_{k,\ell}$ is equal to $\D$ if the $(k,\ell)$ cell of the refined gridding of $\pi$ given by the lines $\L'$ is nonempty, and equal to $\emptyset$ otherwise.  The graph of $\M$ has no edges, and $\M$ is of size at most $2tu\cdot g(f(4mtu))\times 2tu\cdot g(f(4mtu))$.

In order to focus on the most important parts of the proof, which we use again to prove Theorem~\ref{small-classes-gridding}, this sketch has skipped over several details such as why it suffices to consider a single $\N$-gridded permutation $\pi$, and why the collection $\R$ of separated rectangles is independent.  This is remedied in the proof below.

\newenvironment{alternations-gridding-proof}{\medskip\noindent {\it Proof of Theorem~\ref{alternations-gridding}.\/}}{\qed\bigskip}
\begin{alternations-gridding-proof}
Suppose that $\C$ is $\N$-griddable for a $t\times u$ matrix $\N$ whose entries are all subclasses of $\D$, and that $\C$ does not contain any alternations of length $2m$.  It suffices to establish that there exist constants $J$ and $K$, depending only on $\C$, so that every $\pi\in\C$ has an $\M$-gridding for some matrix $\M$ of the desired form and of size at most $J\times K$.  The proof will then follow because there are only finitely many such matrices, and thus, as in the proof of Proposition~\ref{grid-union}, the direct sum of all of them will satisfy the conclusion of the theorem.  To this end choose an $\N$-gridded permutation $\pi\in\C$ and let $\R$ denote the set of separated rectangles for the chosen gridding of $\pi$.

Suppose first that $\R$ contains an independent set of size $4mtu$.  Then at least $4m$ of those rectangles will lie completely within some cell say $(k,\ell)$, of the chosen gridding of $\pi$.  Each of these separated rectangles contains two entries that are separated by an entry in another cell, and thus there is a subset of at least $m$ of these rectangles that are separated by entries in the same direction (direction here meaning one of $\{$left, right, up, down$\}$); we suppose that these entries are all separated by points above them, i.e., each pair is separated by an entry in a cell $(k,\ell^+)$ for some $\ell^+>\ell$.  However, we now have that the leftmost points of these separated rectangles, each together with a point that separated them from their partners, form an alternation of length $2m$.  Therefore, since we have assumed that $\C$ has no alternations of this length, $\R$ cannot contain an independent set of size $4mtu$, and thus by Lemma~\ref{rectangles-lemma}, $\R$ can be sliced by a collection, say $\L$, of at most $f(4mtu)$ vertical and horizontal lines.

Therefore, as we showed before, we can use Lemma~\ref{lem-max-unsliced-ints} to establish an $\M$-gridding of $\pi$ for some matrix $\M$ with an edgeless graph, nonempty cells labeled by $\D$, and of size at most $2g(f(4mtu))\times 2g(f(4mtu))$, completing the proof of the theorem.
\end{alternations-gridding-proof}


The main result of this section, below, extends Theorem~\ref{alternations-gridding} to handle larger classes.  The terms involved are defined immediately after the statement.

\begin{theorem}\label{small-classes-gridding}
Suppose that the permutation class $\C$ is $\D$-griddable, that $\D$ has only finitely many simple permutations and bounded substitution depth, and that $\C$ does not contain arbitrarily long $(3,1)$ or triple alternations.  Then $\C$ is $\M$-griddable for a matrix $\M$ whose cell graph satisfies
\begin{enumerate}
\item[(1)] every vertex is labeled by $\D$, the class of monotone permutations, or is empty,
\item[(2)] every nonisolated vertex is labeled by the class of monotone permutations, and
\item[(3)] there are no connected components containing more than two vertices.
\end{enumerate}
In particular, these conditions hold whenever $\lgr(\C)<1+\sqrt{2}\approx 2.41421$.
\end{theorem}

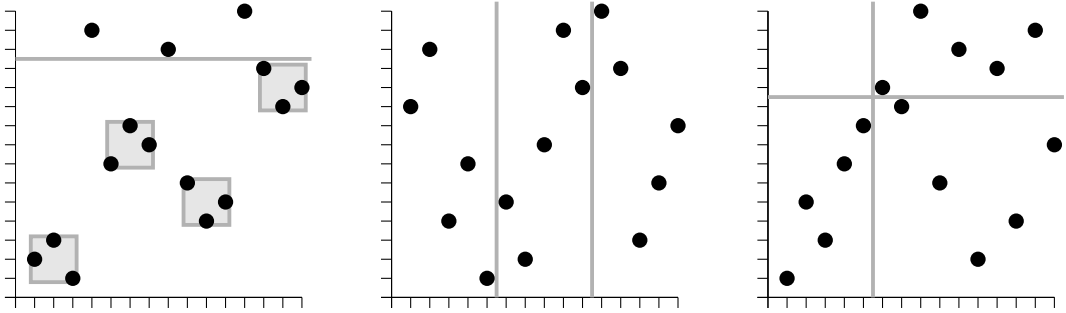
\begin{figure}
\begin{center}
\begin{tabular}{ccccc}
\psset{xunit=0.01in, yunit=0.01in}
\psset{linewidth=0.005in}
\begin{pspicture}(0,0)(150,150)
\psaxes[dy=10,Dy=1,dx=10,Dx=1,tickstyle=bottom,showorigin=false,labels=none](0,0)(150,150)
\psline[linecolor=darkgray,linestyle=solid,linewidth=0.02in](0,125)(155,125)
\psframe[linecolor=darkgray,fillstyle=solid,fillcolor=lightgray,linewidth=0.02in](7,7)(33,33)
\psframe[linecolor=darkgray,fillstyle=solid,fillcolor=lightgray,linewidth=0.02in](47,67)(73,93)
\psframe[linecolor=darkgray,fillstyle=solid,fillcolor=lightgray,linewidth=0.02in](87,37)(113,63)
\psframe[linecolor=darkgray,fillstyle=solid,fillcolor=lightgray,linewidth=0.02in](127,97)(153,123)
\pscircle*(10,20){0.04in}
\pscircle*(20,30){0.04in}
\pscircle*(30,10){0.04in}
\pscircle*(40,140){0.04in}
\pscircle*(50,70){0.04in}
\pscircle*(60,90){0.04in}
\pscircle*(70,80){0.04in}
\pscircle*(80,130){0.04in}
\pscircle*(90,60){0.04in}
\pscircle*(100,40){0.04in}
\pscircle*(110,50){0.04in}
\pscircle*(120,150){0.04in}
\pscircle*(130,120){0.04in}
\pscircle*(140,100){0.04in}
\pscircle*(150,110){0.04in}
\end{pspicture}
&\rule{10pt}{0pt}&
\psset{xunit=0.01in, yunit=0.01in}
\psset{linewidth=0.005in}
\begin{pspicture}(0,0)(150,150)
\psaxes[dy=10,Dy=1,dx=10,Dx=1,tickstyle=bottom,showorigin=false,labels=none](0,0)(150,150)
\psline[linecolor=darkgray,linestyle=solid,linewidth=0.02in](55,0)(55,155)
\psline[linecolor=darkgray,linestyle=solid,linewidth=0.02in](105,0)(105,155)
\pscircle*(10,100){0.04in}
\pscircle*(20,130){0.04in}
\pscircle*(30,40){0.04in}
\pscircle*(40,70){0.04in}
\pscircle*(50,10){0.04in}
\pscircle*(60,50){0.04in}
\pscircle*(70,20){0.04in}
\pscircle*(80,80){0.04in}
\pscircle*(90,140){0.04in}
\pscircle*(100,110){0.04in}
\pscircle*(110,150){0.04in}
\pscircle*(120,120){0.04in}
\pscircle*(130,30){0.04in}
\pscircle*(140,60){0.04in}
\pscircle*(150,90){0.04in}
\end{pspicture}
&\rule{10pt}{0pt}&
\psset{xunit=0.01in, yunit=0.01in}
\psset{linewidth=0.005in}
\begin{pspicture}(0,0)(150,150)
\psaxes[dy=10,Dy=1,dx=10,Dx=1,tickstyle=bottom,showorigin=false,labels=none](0,0)(150,150)
\psline[linecolor=darkgray,linestyle=solid,linewidth=0.02in](0,105)(155,105)
\psline[linecolor=darkgray,linestyle=solid,linewidth=0.02in](55,0)(55,155)
\pscircle*(10,10){0.04in}
\pscircle*(20,50){0.04in}
\pscircle*(30,30){0.04in}
\pscircle*(40,70){0.04in}
\pscircle*(50,90){0.04in}
\pscircle*(60,110){0.04in}
\pscircle*(70,100){0.04in}
\pscircle*(80,150){0.04in}
\pscircle*(90,60){0.04in}
\pscircle*(100,130){0.04in}
\pscircle*(110,20){0.04in}
\pscircle*(120,120){0.04in}
\pscircle*(130,40){0.04in}
\pscircle*(140,140){0.04in}
\pscircle*(150,80){0.04in}
\end{pspicture}
\end{tabular}
\end{center}
\caption{From left to right, a $(3,1)$-alternation, linear triple alternation, and hook triple alternation.}\label{fig-grid-condition}
\end{figure}

We prove Theorem~\ref{small-classes-gridding} in two steps.  First we show that if a class does not have a gridding satisfying condition (2) then it contains arbitrarily long {\it $(3,1)$-alternations\/}, which is our term for permutations that can be divided into two parts $A$ and $B$, say, so that part $A$ consists of nonmonotone intervals of length three, each separated from every other by at least one point in part $B$, and every pair of points in part $B$ is separated by at least one of the intervals of part $A$ (see Figure~\ref{fig-grid-condition}).  In other words, a $(3,1)$-alternation can be obtained by inflating ``one half'' of a horizontal or vertical alternation by nonmonotone intervals of length $3$.  Proposition~\ref{prop-ld-growth} shows that if the permutation class $\C$ contains arbitrarily long $(3,1)$-alternations then $\lgr(\C)\ge 1+\sqrt{2}\approx 2.41421$.

After that, we show that if a class does not satisfy (3) then it contains arbitrarily long {\it triple alternations\/}.  There are two types of these (which are also depicted in Figure~\ref{fig-grid-condition}):
\begin{itemize}
\item A {\it linear triple alternation\/} is one which can be divided into three parts, each with an equal number of points, so that every pair of points in one part is separated by at least one point in each of the other parts.
\item A {\it hook triple alternation\/} is one that can be divided into three parts $A$, $B$, and $C$, again each with an equal number of points, so that no pair of points from $A$ is separated by a point of $C$ or vice versa, but every pair from $A$ or $C$ is separated by at least one point from $B$, and every pair from $B$ is separated by at least one point in $A$ and at least one point in $C$.
\end{itemize}
We establish in Subsection~\ref{subsec-triple-alternations} that if $\C$ contains arbitrarily long triple alternations then $\lgr(\C)\ge 1+\varphi\approx 2.61803$ where $\varphi$ denotes the golden ratio.


\newenvironment{small-classes-gridding-proof}{\medskip\noindent {\it Proof of Theorem \ref{small-classes-gridding}.\/}}{\qed}
\begin{small-classes-gridding-proof}
The beginning of the proof mirrors that of Theorem~\ref{alternations-gridding}.  Suppose that $\C$ is $\N$-griddable for a $t\times u$ matrix $\N$ whose entries are all subclasses of $\D$ and that $\C$ contains neither $(3,1)$-alternations of length $4m$ nor triple alternations of length $3m$.  Following the argument of the previous proof, it suffices to find constants $J$ and $K$, depending only on $\C$, so that every $\pi\in\C$ has an $\M$-gridding for some matrix $\M$ of the desired form and of size at most $J\times K$.

Choose an $\N$-gridded permutation $\pi\in\C$.  We define a {\it $(3,1)$-rectangle\/} to be an axis-parallel rectangle $R$ such that $\pi(R)$ is completely contained in one cell of the chosen gridding, is nonmonotone, and is separated by an entry from another cell of the gridding (see Figure~\ref{fig-fancy-rectangles}).  Let $\R_{(3,1)}$ denote the set of all $(3,1)$-rectangles in the chosen gridding of $\pi$.

\begin{figure}
\begin{center}
\begin{tabular}{ccccc}
\psset{xunit=0.01in, yunit=0.01in}
\psset{linewidth=0.005in}
\begin{pspicture}(0,0)(100,100)
\psframe[linecolor=darkgray,linewidth=0.02in](0,0)(100,100)
\psframe[linecolor=darkgray,fillstyle=solid,fillcolor=lightgray,linewidth=0.02in](7,7)(73,53)
\psline[linecolor=darkgray,linestyle=solid,linewidth=0.02in](0,70)(100,70)
\pscircle*(10,50){0.04in}
\pscircle*(40,10){0.04in}
\pscircle*(70,30){0.04in}
\pscircle*(55,80){0.04in}
\end{pspicture}
&\rule{10pt}{0pt}&
\psset{xunit=0.01in, yunit=0.01in}
\psset{linewidth=0.005in}
\begin{pspicture}(0,0)(100,100)
\psframe[linecolor=darkgray,linewidth=0.02in](0,0)(100,100)
\psframe[linecolor=darkgray,fillstyle=solid,fillcolor=lightgray,linewidth=0.02in](7,17)(23,83)
\psline[linecolor=darkgray,linestyle=solid,linewidth=0.02in](35,0)(35,100)
\psline[linecolor=darkgray,linestyle=solid,linewidth=0.02in](75,0)(75,100)
\pscircle*(10,80){0.04in}
\pscircle*(20,20){0.04in}
\pscircle*(50,40){0.04in}
\pscircle*(90,60){0.04in}
\end{pspicture}
&\rule{10pt}{0pt}&
\psset{xunit=0.01in, yunit=0.01in}
\psset{linewidth=0.005in}
\begin{pspicture}(0,0)(100,100)
\psframe[linecolor=darkgray,linewidth=0.02in](0,0)(100,100)
\psframe[linecolor=darkgray,fillstyle=solid,fillcolor=lightgray,linewidth=0.02in](7,57)(33,93)
\psline[linecolor=darkgray,linestyle=solid,linewidth=0.02in](0,50)(100,50)
\psline[linecolor=darkgray,linestyle=solid,linewidth=0.02in](50,0)(50,100)
\pscircle*(10,60){0.04in}
\pscircle*(30,90){0.04in}
\pscircle*(25,20){0.04in}
\pscircle*(80,75){0.04in}
\end{pspicture}
\end{tabular}
\end{center}
\caption{From left to right, a $(3,1)$-rectangle and two doubly separated rectangles.}\label{fig-fancy-rectangles}
\end{figure}
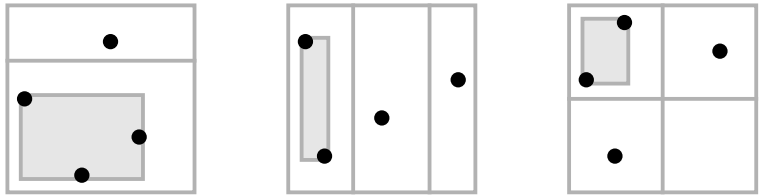

If $\R_{(3,1)}$ contains an independent set of size at least $8mtu$ then, as in the proof of Theorem~\ref{alternations-gridding}, at least $8m$ of those rectangles must lie completely within some cell of the gridding of $\pi$ and then at least $2m$ of those rectangles must be separated in the same direction; let us suppose this direction is ``up.''  However, we can now find a $(3,1)$-alternation of length $4m$ contained in $\pi$: list the $(3,1)$-rectangles from left to right and take a nonmonotone permutation of length $3$ from the $1$st rectangle, a point separating the $2$nd rectangle, a nonmonotone permutation of length $3$ from the $3$rd rectangle, a point separating the $4$th rectangle, and so on.  Therefore, as we have assumed that $\C$ does not contain such permutations, $\R_{(3,1)}$ cannot contain an independent set of size $8mtu$, and thus $\R_{(3,1)}$ can be sliced by a set $\L_{(3,1)}$ of $f(8mtu)$ vertical and horizontal lines by Lemma~\ref{rectangles-lemma}.

Now we define a {\it doubly separated rectangle\/} to be an axis-parallel rectangle $R$ such that $\pi(R)$ is completely contained in one cell of the chosen gridding of $\pi$ and contains (at least) two entries that are separated twice, by two entries lying in different cells from each other and from $R$.  Figure~\ref{fig-fancy-rectangles} shows two examples.  Let $\R_{(2,1,1)}$ denote the set of doubly separated rectangles in $\pi$.

Suppose that $\R_{(2,1,1)}$ contains an independent set, $\I$, of size $mtu(t+u)^2$.  At least $m(t+u)^2$ of these rectangles must lie in the same cell of the chosen gridding of $\pi$; let $\R_{(2,1,1)}'$ denote a set of such doubly separated rectangles.  Each doubly separated rectangle is separated by two entries in different cells, and so in $\R_{(2,1,1)}'$ we can find a collection, say $\R_{(2,1,1)}''$, of at least $m$ doubly separated rectangles whose separating points lie in the same two cells.  However, these rectangles and separating entries give rise to a triple alternation of length at least $3m$, a contradiction.  Therefore we may assume that $\R_{(2,1,1)}$ does not contain an independent set of size $mtu(t+u)^2$, and thus by Lemma~\ref{rectangles-lemma}, it can be sliced by a collection $\L_{(2,1,1)}$ of at most $f(mtu(t+u)^2)$ horizontal and vertical lines.

The rest of the proof is similar in spirit to that of Theorem~\ref{alternations-gridding}, with some additional complications.  Set $\L=\L_{(3,1)}\cup\L_{(2,1,1)}$.  Lemma~\ref{lem-max-unsliced-ints} shows that these lines create at most $g(f(8mtu)+f(mtu(t+u)^2))$ maximal unsliced intervals in each cell of the $\N$-gridding of $\pi$, giving at most $tu\cdot g(f(8mtu)+f(mtu(t+u)^2))$ maximal unsliced intervals in total.  Let $\HH$ denote the rectangular hulls of these maximal unsliced intervals.

As in the proof of Theorem~\ref{alternations-gridding}, no two hulls in $\HH$ which come from the same cell of the $\N$-gridding of $\pi$ may overlap, as the hulls are formed from intervals.  However, unlike the situation in that proof, hulls in $\HH$ from different cells may overlap, but since none of these hulls may contain a $(3,1)$-rectangle, any pair of overlapping hulls must both be monotone.  Moreover, because no hull in $\HH$ may contain a doubly separated rectangle, it follows that if two hulls with at least two points overlap with each other, then neither may overlap with any other hull.  Finally note that a hull with a single point may overlap at most two other hulls: it cannot overlap another singleton hull by the definition of hulls, and if it were to overlap two hulls horizontally (or vertically), then those two cells would themselves overlap.

We now consider a second slicing, where we let $\L'$ denote the set of at most $4tu\cdot g(f(8mtu)+f(mtu(t+u)^2))$ vertical and horizontal lines which coincide with the sides of the hulls in $\HH$.  From our observations about the possible interaction of hulls in $\HH$, we see that the only way in which a line in $\L'$ could slice through a previously unsliced interval would be if that interval were monotone, in which case the new line would break the interval into $2$ maximal unsliced intervals.  Furthermore, we have also seen that this can happen at most twice for each hull in $\HH$ (and indeed, at most once except when the hull contains a single point).  Therefore, even after being sliced by all the lines in $\L'$, the $\N$-gridding of $\pi$ contains at most $3tu\cdot g(f(8mtu)+f(mtu(t+u)^2))$ maximal unsliced intervals.


At this stage we let $\HH'$ denote the rectangular hulls of all maximal unsliced intervals of the $\N$-gridding of $\pi$ with respect to the new lines $\L'$ and let $\S$ denote the set of at most $12tu\cdot g(f(8mtu)+f(mtu(t+u)^2))$ vertical and horizontal lines which coincide with the sides of these hulls.  By construction, no hull in $\HH'$ may overlap with more than one other hull, and this may occur only if both hulls are monotone.  Now let $\M$ denote the matrix in which $\M_{k,\ell}$ is equal to $\D$ if the $(k,\ell)$ cell of the refined gridding of $\pi$ given by the lines $\S$ contains a nonmonotone permutation, to the set of increasing permutations if this cell is increasing, to the set of decreasing permutations if it is decreasing, and to $\emptyset$ if it is devoid of points.  This matrix $\M$ gives us a gridding of $\pi$ satisfying the conditions required by the theorem, completing the proof.
\end{small-classes-gridding-proof}

\section{The Growth Rates Below $\kappa$}\label{sec-small-growth-rates}

With the structure of small permutation classes established in the previous two sections, we now turn our attention to properties such as pwo, finite bases, and of course growth rates.

\begin{proposition}\label{sub-kappa-pwo}
Let $\D$ be a pwo permutation class and suppose that the matrix $\M$ satisfies the conclusion of Theorem~\ref{small-classes-gridding} with respect to $\D$.  Then $\Grid(\M)$ (and thus each of its subclasses) is also pwo.  In particular, by Proposition~\ref{prop-WOk-gridding} and Theorem~\ref{small-classes-gridding}, every permutation class with lower growth rate less than $\kappa$ is pwo.
\end{proposition}
\begin{proof}
By Proposition~\ref{grid-pwo-component} it suffices to show that every connected component of the cell graph of $\M$ is pwo.  By assumption, these are either $\D$ or edges in which both cells are labeled by monotone classes.  The proof is then completed by noting that both classes are pwo (by assumption and Theorem~\ref{pwo-forest}, respectively).
\end{proof}

It is now not difficult to see that small classes must have finite bases.

\begin{theorem}\label{sub-kappa-finite-basis}
Every permutation class with lower growth rate less than $\kappa$ is finitely based.
\end{theorem}
\begin{proof}
Suppose, to the contrary, that there is a permutation class $\C$ with $\lgr(C)<\kappa$ and an infinite basis, $B$.  Choose an infinite subset $A\subseteq B$ containing at most one permutation of each length and let $\C'=\C\cup A$.  Clearly $\lgr(\C')=\lgr(C)<\kappa$, but then Proposition~\ref{sub-kappa-pwo} implies that $\C'$ is pwo, a contradiction.
\end{proof}

Note, however, that it does not follow from Proposition~\ref{sub-kappa-pwo} that there are only countably many small permutation classes.  This is our next result, and answer to Question 2 of the Introduction.

\begin{theorem}\label{theorem-kappa}
There are only countably many permutation classes with lower growth rate less than $\kappa$.
\end{theorem}
\begin{proof}
We have already noted that each of these classes lies in $\Grid(\M)$ for some matrix $\M$ which satisfies the conclusion of Theorem~\ref{small-classes-gridding} with respect to $\S(\O_k)$ for some $k$.  Note that for every $k$ there are only countably many matrices of this form.  Now consider a fixed matrix $\M$ of this form and let $\gimel_\M$ denote the set of permutation classes contained in $\Grid(\M)$.  Proposition~\ref{sub-kappa-pwo} shows that $\Grid(\M)$ is pwo, so $\gimel_\M$ is countable by Proposition~\ref{pwo-subclasses}.  Thus we have that the set of permutation classes with lower growth rate less than $\kappa$ is the countable union of a countable union of countable sets $\bigcup_{k\ge 1}\bigcup_{\M}\gimel_\M$ where the inner union is over all matrices $\M$ satisfying Theorem~\ref{small-classes-gridding} with respect to $\S(\O_k)$, proving the theorem.
\end{proof}

As we now have, by Proposition~\ref{sub-kappa-pwo}, that small classes are pwo, we can apply Propositions~\ref{pwo-subclasses-dcc}, \ref{pwo-atomic-gr}, and \ref{pwo-grid-irreducible} to study their possible growth rates.

\begin{theorem}\label{gr-atomic-grid-irreducible}
If the permutation class $\C$ satisfies $\lgr(\C)<\kappa$ then $\gr(\C)$ exists and is equal to $0$, $1$, $2$, or the growth rate of a subclass that is either sum or skew sum complete.
\end{theorem}
\begin{proof}
First set $\C^0=\C$ and choose an atomic subclass $\A^0\subseteq\C^0$ for which $\ugr(\A^0)=\ugr(\C^0)$; such a class exists by Proposition~\ref{pwo-atomic-gr} because $\C^0$ is pwo.  As $\A^0$ is itself pwo, Proposition~\ref{pwo-grid-irreducible} shows that $\A^0$ is $\D^0$-griddable for a grid irreducible class $\D^0\subseteq\A^0$.  To complete the base case, note that $\lgr(\A^0)\le\lgr(\C^0)<1+\sqrt{2}$ and so $\A^0$ is $\M^0$-griddable for a matrix $\M^0$ which satisfies the conditions of Theorem~\ref{small-classes-gridding} with respect to $\D^0$.

For each $i\ge 1$, let $\C^i$ denote a restriction of $\A^{i-1}$ to a connected component of $\M^{i-1}$ which satisfies $\ugr(\C^i)=\ugr(\A^{i-1})$, noting that such a class exists by Proposition~\ref{grid-gr-component}.  If this connected component is of size two then, by our conditions on $\M^{i-1}$, $\C^i$ is contained in a $1\times 2$ or $2\times 1$ monotone grid class.  This case is handled by Proposition~\ref{prop-gr-mono-vector}, which shows that subclasses of monotone grid classes of vectors have integral growth rates, and therefore subclasses of $1\times 2$ or $2\times 1$ monotone grid classes have growth rates of $0$, $1$, or $2$.

Let us therefore suppose this connected component consists of a single cell of $\M^{i-1}$ and thus $\C^i\subseteq\D^{i-1}$.  Now choose an atomic class $\A^i\subseteq\C^i$ with $\ugr(\A^i)=\ugr(\C^i)$ and a grid irreducible subclass $\D^i\subseteq\A^i$ so that $\A^i$ is $\D^i$-griddable.  Finally, choose a matrix $\M^i$ satisfying the conditions of Theorem~\ref{small-classes-gridding} with respect to $\D^i$ for which $\A^i$ is $\M^i$-griddable.

In this process we create a descending chain of classes
$$
\C=\C^0\supseteq \A^0\supseteq\D^0\supseteq\C^1\supseteq\A^1\supseteq\D^1\supseteq\cdots,
$$
all with identical upper growth rates.  As $\C$ is pwo this chain must terminate by Proposition~\ref{pwo-subclasses-dcc}; let $h$ denote the least integer such that $\C^{h+1}=\C^{h}$.

We have $\C^{h}=\A^{h}=\D^{h}=\C^{h+1}$, so $\C^h$ is both grid irreducible and atomic, and thus, by Proposition~\ref{atomic-grid-irreducible}, either sum or skew complete.   By construction we also have that $\ugr(\C^h)=\ugr(\C)$.  Proposition~\ref{arratia-gr} shows that $\gr(\C^h)$ exists, and thus we must have $\lgr(\C)=\ugr(\C)=\gr(\C^h)$, so $\gr(\C)$ exists and is equal to $\gr(\C^h)$, completing the proof.
\end{proof}

Theorem~\ref{gr-atomic-grid-irreducible} reduces our search for growth rates below $\kappa$ to sum complete classes (by symmetry).  As we now demonstrate, these sum complete classes can be enumerated via Proposition~\ref{enum-oplus-completion}, using the results of Subsection~\ref{subsec-sum-indecomps}.

\begin{theorem}\label{small-growth-rates}
The sub-$\kappa$ growth rates of permutation classes consist precisely of $0$ and roots of the following families of polynomials:
$$
\begin{array}{r|l|l}
&\mbox{sequence of sum indecomposables}&\mbox{growth rate $=$ the largest root of}\\
\hline
\mbox{(I)}&(1\times k)&x^{k+1}-2x^k+1\\
\mbox{(II)}&(1\times\infty)&x-2\\
\mbox{(III)}&(1,1,2\times k, 1\times\ell)&(x^3-2x^2-1)x^{k+\ell}+x^\ell+1\\
\mbox{(IV)}&(1,1,2\times k, 1\times\infty)&(x^3-2x^2-1)x^k+1\\
\mbox{(V)}&(1,1,2,3)\mbox{ and }(1,1,2,3,1)&x^4-x^3-x^2-2x-3\\
&&\mbox{and }x^5-x^4-x^3-2x^2-3x-1\\
\mbox{(VI)}&(1,1,3)\mbox{ and }(1,1,3,1)&x^3-x^2-x-3\\
&&\mbox{and }x^4-x^3-x^2-3x-1
\end{array}
$$
\end{theorem}

The smallest of these growth rates that is greater than $2$, $\nu\approx 2.06599$, occurs when $k=\ell=1$ in family (III).  Furthermore, the growth rates of type (III) accumulate at growth rates of type (IV) which themselves accumulate at $\kappa$, so $\kappa$ is the smallest accumulation point of accumulation points of growth rates of permutation classes.

\begin{proof}
By Theorem~\ref{gr-atomic-grid-irreducible} it suffices to consider small monotone grid classes (which have integral growth rates by Proposition~\ref{prop-gr-mono-vector}) and (by symmetry) sum complete classes, so suppose that $\C$ is a sum complete class.  Let $s_n$ denote the number of sum indecomposable permutations of length $n\ge 1$ contained in $\C$, so the generating function for $\C$ is $1/(1-\sum s_nx^n)$ by Proposition~\ref{enum-oplus-completion}.  We say that the sequence $(s_n)$ is {\it small\/} if it leads to a growth rate less than $\kappa$ and {\it large\/} otherwise.  We also write $(s_n)\le (t_n)$ if $s_n\le t_n$ for all $n\ge 1$; note that if $(s_n)$ is large and $(s_n)\le (t_n)$ then $(t_n)$ is also large.

We first show that the sequence
$$
(1, 1, \underbrace{2, \dots, 2}_k, 4),
$$
which we abbreviate as $(1,1,2\times k,4)$, is large for all $k\ge 0$ (although, when $k=0$ this sequence is impossible for a permutation class to achieve because there are only $3$ sum indecomposable permutations of length $3$).  The generating function for the corresponding class is, by Proposition~\ref{enum-oplus-completion},
$$
f
=
\frac{1}{1-x-x^2-2x^3-\cdots-2x^{k+2}-4x^{k+3}}
=
\frac{1-x}{1-2x-x^3-2x^{k+3}+4x^{k+4}}.
$$
By Pringsheim's Theorem, the growth rate of this class is $1/r$ where $r$ denotes the smallest positive root of the denominator of $f$.  Equivalently, setting $z=1/x$ and multiplying by $z^{k+4}$, the growth rate of this class is equal to the largest positive root of
\begin{eqnarray*}
h
&=&
z^{k+4}\left(1-2(1/z)-(1/z)^3-2(1/z)^{k+3}+4(1/z)^{k+4}\right)\\
&=&
z^{k+4}-2z^{k+3}-z^{k+1}-2z+4\\
&=&4-2z-z^{k+1}\left(1+2z^2-z^3\right).
\end{eqnarray*}
Since $h(\kappa)=4-2\kappa<0$ and the leading coefficient of $g$ is positive, $g$ has a real root greater than $\kappa$, establishing that the sequence $(1,1,2\times k,4)$ is large.  Therefore if any sum complete class has $4$ or more sum indecomposable of the same length, then Proposition~\ref{prop-2-sum-indecomps} shows that its growth rate is greater than $\kappa$.

A similar computation shows that the sequence $(1,1,2,3\times k)$ is large for all $k\ge 2$.  By Proposition~\ref{prop-3-sum-indecomps}, if any sum complete class has $3$ sum indecomposable permutations of length $n\ge 5$, then it has at least $3$ sum indecomposable permutations of every length between $4$ and $n$, and by Proposition~\ref{prop-2-sum-indecomps}, it then has $2$ sum indecomposable permutations of length $3$.  Therefore, this computation shows that no small sum complete class can have $3$ or more sum indecomposable permutations of the same length $n\ge 5$.  It remains to consider the classes with $3$ sum indecomposable permutations of length $4$.  Again, similar computations show that the sequences $(1,1,2\times k,3,1,1)$ and $(1,1,2\times k, 3,2)$ are large for all $k\ge 0$.  Finally, the sequence $(1,1,3,1,1)$ leads to a growth rate of precisely $\kappa$, and the sequence $(1,1,3,2)$ is large.

Of the classes with at most two sum indecomposable permutations of each length, the sequence $(1,1,2\times\infty)$ leads to a growth rate of precisely $\kappa$, so all other sequences of this form are small.

These remarks show that if $\C$ is a small sum complete class with $s_n$ sum indecomposable permutations of length $n$ then:
\begin{itemize}
\item $s_n\le 3$ for all $n$,
\item $s_n\le 2$ for all $n\ge 5$,
\item $s_n\le 1$ for all sufficiently large $n$,
\item if $s_3=3$ then $s_4\le 1$ and $s_5=0$,
\item if $s_4=3$ then $s_3=2$, $s_5=1$, and $s_n=0$ for all $n\ge 6$.
\end{itemize}
This leaves only the few possibilities for small sequences listed in the statement of the theorem, and so all that remains is to exhibit permutation classes with such sequences of sum indecomposable permutations.  We list the sum indecomposable permutations of such classes below.
$$
\begin{array}{r|l|l}
&\mbox{sequence of sum indecomposables}&\mbox{growth rate $=$ the largest root of}\\
\hline
\mbox{(I)}&(1\times k)&\{n\cdots 21 : 1\le n\le k\}\\
\mbox{(II)}&(1\times\infty)&\{n\cdots 21 : n\ge 1\}\\
\mbox{(III)}&(1,1,2\times k, 1\times\ell)&\{n\cdots 21 : 1\le n\le k+\ell+2\}\\
&&\cup\{1\ominus 12\cdots (n-1) : 3\le n\le k+2\}\\
\mbox{(IV)}&(1,1,2\times k, 1\times\infty)&\{n\cdots 21 : n\ge 1\}\\
&&\cup\{1\ominus 12\cdots (n-1) : 3\le n\le k+2\}\\
\mbox{(V)}&(1,1,2,3)\mbox{ and }(1,1,2,3,1)&\{1,21,312,321,4123,4132,4321\}\\
&&\mbox{and }\{1,21,312,321,4123,4132,4321,54321\}\\
\mbox{(VI)}&(1,1,3)\mbox{ and }(1,1,3,1)&\{1,21,231,312,321\}\\
&&\mbox{and }\{1,21,231,312,321,4321\}
\end{array}
$$
Note that this is not an exhaustive list of possible sets of sum indecomposable permutations, but rather one example for each sequence.
\end{proof}

\section{Concluding Remarks}

\begin{figure}
\begin{footnotesize}
\begin{center}
\begin{tabular}{ccccccccccc}
\psset{xunit=0.007in, yunit=0.007in}
\begin{pspicture}(0,0)(20,30)
\pscircle*(10,15){0.04in}
\rput[c](10,0){$1$}
\end{pspicture}
&\rule{1pt}{0pt}&
\psset{xunit=0.007in, yunit=0.007in}
\psset{linewidth=0.02in}
\begin{pspicture}(0,0)(50,30)
\pscircle*(10,15){0.04in}
\pscircle*(40,15){0.04in}
\pscurve(10,15)(20,25)(30,25)(40,15)
\rput[c](10,0){$1$}
\rput[c](40,0){$2$}
\end{pspicture}
&\rule{1pt}{0pt}&
\psset{xunit=0.007in, yunit=0.007in}
\psset{linewidth=0.02in}
\begin{pspicture}(0,0)(80,30)
\pscircle*(10,15){0.04in}
\pscircle*(40,15){0.04in}
\pscircle*(70,15){0.04in}
\pscurve(10,15)(25,25)(55,25)(70,15)
\rput[c](10,0){$1$}
\rput[c](40,0){$2$}
\rput[c](70,0){$3$}
\end{pspicture}
&\rule{1pt}{0pt}&
\psset{xunit=0.007in, yunit=0.007in}
\psset{linewidth=0.02in}
\begin{pspicture}(0,0)(110,30)
\pscircle*(10,15){0.04in}
\pscircle*(40,15){0.04in}
\pscircle*(70,15){0.04in}
\pscircle*(100,15){0.04in}
\pscurve(10,15)(40,25)(70,25)(100,15)
\rput[c](10,0){$1$}
\rput[c](40,0){$2$}
\rput[c](70,0){$3$}
\rput[c](100,0){$4$}
\end{pspicture}
&\rule{1pt}{0pt}&
\psset{xunit=0.007in, yunit=0.007in}
\psset{linewidth=0.02in}
\begin{pspicture}(0,0)(110,30)
\pscircle*(10,15){0.04in}
\pscircle*(40,15){0.04in}
\pscircle*(70,15){0.04in}
\pscircle*(100,15){0.04in}
\pscurve(10,15)(25,25)(55,25)(70,15)
\pscurve(40,15)(55,25)(85,25)(100,15)
\rput[c](10,0){$1$}
\rput[c](40,0){$2$}
\rput[c](70,0){$3$}
\rput[c](100,0){$4$}
\end{pspicture}
&\rule{1pt}{0pt}&
\psset{xunit=0.007in, yunit=0.007in}
\psset{linewidth=0.02in}
\begin{pspicture}(0,0)(140,30)
\pscircle*(10,15){0.04in}
\pscircle*(40,15){0.04in}
\pscircle*(70,15){0.04in}
\pscircle*(100,15){0.04in}
\pscircle*(130,15){0.04in}
\pscurve(10,15)(50,25)(90,25)(130,15)
\rput[c](10,0){$1$}
\rput[c](40,0){$2$}
\rput[c](70,0){$3$}
\rput[c](100,0){$4$}
\rput[c](130,0){$5$}
\end{pspicture}
\end{tabular}
\end{center}
\end{footnotesize}
\caption{The sum indecomposable ordered graphs used to make a hereditary property of ordered graphs with growth rate $\approx 2.03166$.}\label{fig-ordered-graphs}
\end{figure}
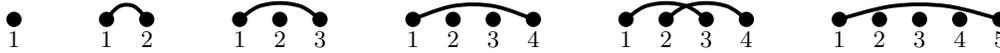

\noindent{\bf Ordered graphs.}
Growth rates of hereditary properties of other types of combinatorial structure (analogues of permutation classes) have received considerable attention recently.  The case of ordered graphs, studied by Balogh, Bollob\'as, and Morris~\cite{balogh:hereditary-prop:ordgraphs}, bears a striking resemblance to the permutation case.  Note that while the sub-$2$ growth rates of permutation classes and of hereditary properties of ordered graphs are identical, this is not the case above $2$; there is a hereditary property of ordered graphs whose growth rate is the largest real root of $1+2x+x^2+x^3+x^4-x^5$, approximately $2.03166$ (to see this, note that the definition of direct sum extends naturally to ordered graphs, and take the hereditary property of ordered graphs whose sum indecomposable graphs are shown in Figure~\ref{fig-ordered-graphs}), while we have proved that this is not a growth rate of any permutation class.  Balogh, Bollob\'as, and Morris~\cite{balogh:hereditary-prop:ordgraphs} conjecture that this is the smallest growth rate above $2$ for ordered graphs.

Conversely, every growth rate of a permutation class is the growth rate of a hereditary property of ordered graphs; simply consider ordered versions of the graphs $G_\pi$ for permutations $\pi$ in the class.

\minisec{Growth rates beyond $\mathbf{\kappa}$}
Balogh, Bollob\'as, and Morris~\cite{balogh:hereditary-prop:ordgraphs} had conjectured that the set of growth rates of ordered graphs (and thus also of permutation classes) has no accumulation points from above and contains only integers and irrational algebraic numbers.  Albert and Linton~\cite{albert:growing-at-a-pe:} disproved these conjectures by showing that there is a perfect uncountable set (i.e., a set equal to its accumulation points) of growth rates of permutation classes between the unique positive root of $3+4x+2x^2+2x^3+x^4-x^5$, $\approx 2.47665$, and $3$.

Vatter~\cite{vatter:permutation-cla} has since proved that the set of growth rates of permutation classes contains every number greater than or equal to the unique real root of $x^5-2x^4-2x^2-2x-1$, $\approx 2.48187$.  In some sense this leaves us close to the ultimate aim of this line of research:

\begin{problem}
Characterize all growth rates of permutation classes.
\end{problem}

However, there are many technical difficulties with extending our techniques to deal with larger permutation classes.  The first of these is partial well order, which need not hold for permutation classes of growth rate $\kappa$ or more, but is essential to the proof of Theorem~\ref{gr-atomic-grid-irreducible}.  Even if these difficulties were overcome, and the problem simplified to characterizing growth rates of sum complete permutation classes, we would still need a significant extension of Propositions~\ref{prop-3-sum-indecomps}.

\minisec{The existence of growth rates}
Of particular irritation is the fact that arbitrary permutation classes are not known to have growth rates.  We have shown (Theorem~\ref{gr-atomic-grid-irreducible}) that every small permutation class $\C$ has a growth rate, and that this growth rate is either $0$, $1$, $2$, or the growth rate of the largest sum or skew sum complete subclass in $\C$.  However, this behavior surely cannot extend much further, especially for non-pwo permutation classes.  Therefore the general problem of proving that permutation classes have growth rates remains wide open.




\minisec{Enumeration below $\kappa$}
Thus far we have ignored the exact enumeration problem, not out of neglect but rather complete ignorance.  Permutation classes of growth rate $1$ are known to have eventually polynomial enumeration (first established by Kaiser and Klazar~\cite{kaiser:on-growth-rates:}, this result follows rather quickly from the results presented here, particularly Theorem~\ref{alternations-gridding} applied to the monotone griddings that such classes must have; see \cite{huczynska:grid-classes-an:} for the missing details).

Moving beyond the very small classes, Albert, Linton, and Ru\v{s}kuc~\cite{albert:the-insertion-e:} have introduced a correspondence between permutation classes and formal languages, known as the insertion encoding, in which every finitely based class which does not contain arbitrarily long vertical alternations (see Footnote~\ref{fn-alternations} for the definition of vertical alternations) corresponds to a regular language (which can be computed by the Maple package {\sc InsEnc} described in Vatter~\cite{vatter:finding-regular:}).  As no class with a growth rate below $2$ can contain arbitrarily long alternations (let alone vertical alternations), these all have rational generating functions.  It is thus natural to ask for the smallest class with an nonrational generating function.

\begin{conjecture}\label{smallest-nonrational}
Every permutation class with growth rate less than $\kappa$ has a rational generating function.
\end{conjecture}

Note that there are classes with growth rate at most $\kappa$ which have not only nonrational, but nonholonomic, generating functions%
\footnote{\addtocounter{theorem}{1}%
{\bf Proposition~\arabic{section}.\arabic{theorem}}
(Atkinson and Stitt~\cite{atkinson:restricted-perm:wreath}; Murphy~{\cite[Chapter 9]{murphy:restricted-perm:}}).
{\it If the permutation class $\C$ contains an infinite antichain then it also contains a subclass with a nonholonomic generating function.\/}

\newenvironment{fn-proof}{\medskip\noindent {\it Proof.\/}}{\qed}
\begin{fn-proof}
Choose an infinite antichain $A\subseteq\C$ that has at most one element of each length.  If $A_1\neq A_2$ are two subsets of $A$ then the two subclasses $\C\cap\Av(A_1)$ and $\C\cap\Av(A_2)$ have different enumerations.  Because $A$ is infinite, this gives $2^{\aleph_0}$ different generating functions.  Now notice that if $f$ is a holonomic generating function for a permutation class then the recurrence satisfied by $f$ may be chosen to have integral coefficients and integral initial conditions, and so there are only countably many holonomic generating functions for permutation classes.
\end{fn-proof}}.
However, our belief in Conjecture~\ref{smallest-nonrational} is based on more than coincidence.  Suppose that $\C$ is a permutation class with $\gr(\C)<\kappa$, so $\C$ is $\M$-griddable for a matrix $\M$ satisfying the conditions of Theorem~\ref{small-classes-gridding}.  Then the insertion encoding shows that every restriction of $\C$ to a connected component of $\M$ has a rational generating function (possibly by considering a symmetry)%
\footnote{Also, as we remarked at the end of Section~\ref{sec-subst-gridding}, every cell of this gridding has a rational generating function, since the cells are labeled by classes with finitely many simple permutations and bounded substitution depth.}%
.  The only stumbling block is that permutations often have multiple $\M$-griddings, but they cannot have {\it too\/} many, and this is exactly the sort of problem one expects to be able to handle with regular languages.

\appendix
\section{Calculations}

In the calculations that follow we make frequent use of regular languages, and the reader is referred to Hopcroft, Motwani, and Ullman~\cite{hopcroft:introduction-to:a} for all concepts and notation not defined here.  Given a finite alphabet $A$, we denote by $A^*$ the {\it free monoid over $A$\/}.  This is the regular language consisting of all (possibly empty) words (or, sequences) over $A$, i.e., $A^*=\{a_1\cdots a_n : \mbox{$a_i\in A$ for all $i$}\}$.  We further define $A^+$ as the language of all nonempty words over $A$.  Given two regular languages $K$ and $L$ we denote their union by $K\cup L$ and their concatenation as $KL$ (the set of all words $k\ell$ where $k\in K$ and $\ell\in L$).

\subsection{Increasing Oscillations and Their Substitution Completion}\label{app-oscillations}

The results of Section~\ref{sec-subst-gridding} require us to show that small permutation classes have $\S(\O)$-griddings, where $\O$ denotes the class of all permutations contained in an oscillation.  In order to establish this we first review a basic fact about substitution completions.

\begin{proposition}[Albert and Atkinson~\cite{albert:simple-permutat:}]\label{substitution-completion-basis}
The basis of the substitution completion of the permutation class $\C$, $\S(\C)$, consists of all minimal simple permutations not contained in $\C$.
\end{proposition}
\begin{proof}
Suppose first that $\beta$ is a nonsimple basis element of $\S(\C)$ and express $\beta$ as the inflation $\sigma[\alpha_1,\dots,\alpha_m]$ where $\sigma$ is simple.  If $\sigma$ or one of the $\alpha_i$ is not contained in $\S(\C)$ then it contains a basis element of $\S(\C)$, so $\beta$ is not a basis element of $\S(\C)$, a contradiction.  Thus $\beta$ lies in $\S(\C)$, again a contradiction.  Therefore we conclude that the basis elements of $\S(\C)$ are all simple, and thus they are the minimal simple permutations not contained in $\S(\C)$ by the definition of basis, and the proposition follows by observing that $\S(\C)$ and $\C$ contain the same set of simple permutations.
\end{proof}

This proposition allows us to compute basis of the substitution completion of $\O$:

\begin{proposition}\label{prop-basis-WO}
The basis of $\S(\O)$ consists of $25314$, $41352$, $246153$, $251364$, $314625$, $351624$, $415263$, and every symmetry of one of these permutations.
\end{proposition}
\begin{proof}
Let $B$ denote the basis specified in the statement of the theorem.  It can be checked that each of the $7$ permutations listed are basis elements of $\S(\O)$, as they are simple and do not lie in $\O$.  Since $\O$ is closed under all eight permutation class symmetries, $\S(\O)$ --- and thus the basis of $\S(\O)$ --- must be as well.  Therefore it suffices to prove that every basis element of $\S(\O)$ is a symmetry of one of the $7$ basis elements listed, that is, an element of $B$.  By Proposition~\ref{substitution-completion-basis}, this amounts to proving that every simple permutation not contained in $\O$ contains an element of $B$.  Let $\pi$ denote a simple permutation not contained in $\O$.  We prove the fact by induction on the length, $n$, of $\pi$; as it is easy to check for $n\le 6$, we will assume that $n\ge 7$.

If $\pi$ is a simple parallel alternation then $\pi$ contains at least one of $246135$, $362514$, $415263$, or $531642$, which are all symmetries of $415263$, so $\pi$ contains an element of $B$, as desired.

Thus we may assume that $\pi$ is not a parallel alternation, and thus by Theorem~\ref{thm-schmerl-trotter}, $\pi$ contains a simple permutation of length $n-1$.  Label the indices of this simple permutation $1\le i_1<\cdots<i_{n-1}\le n$ and let $i=[n]\setminus\{i_1,\dots,i_{n-1}\}$ denote the index of the missing entry.  By the minimality of $\pi$, these entries must be order isomorphic to an oscillation, and by symmetry we may assume that they are order isomorphic to $4,1,6,3,\dots,2k+2,2k-1$ if $n-1$ is even or to $4,1,6,3,\dots,2k,2k-3,2k-1$ if $n-1$ is odd.

It is clear that if $i<i_1$ or $i>i_3$ then the permutation obtained from $\pi$ by removing the entry in position $i_2$ must still be simple and not order isomorphic to an oscillation; thus it must contain an element of $B$ by induction, so $\pi$ does as well.  This also occurs if $\pi(i)>\pi(i_5)$, so we may assume that $i_1<i<i_3$ and $\pi(i)<\pi(i_5)$.  In this case, the permutation obtained from $\pi$ by removing the entry in position $i_{n-2}$ if $n$ is even and $i_{n-1}$ if $n$ is odd is simple and not order isomorphic to an oscillation, and thus we are done again by induction, completing the proof.
\end{proof}

As a consequence of this basis result, we see that all small classes have $\S(\O)$-griddings.

\begin{proposition}\label{prop-WO-gridding}
Every class with lower growth rate at most the unique positive root of $1+3x+3x^2+2x^3+x^4+x^5-x^6$, $\approx 2.24409$, has an $\S(\O)$-gridding.
\end{proposition}
\begin{proof}
Theorem~\ref{gridding-characterization} shows that the permutation class $\C$ fails to have an $\S(\O)$-gridding if and only if it contains either $\bigoplus\beta$ or $\bigominus\beta$ for some basis element $\beta$ of $\S(\O)$.  Therefore we need only compute (via Proposition~\ref{enum-oplus-completion}) the growth rates of $\bigoplus\beta$ and $\bigominus\beta$ for each basis element $\beta$ specified in the preceding proposition.  The smallest such classes are the sum completion of $251364$ and the skew sum completion of its reverse.  The generating function for the nonempty sum indecomposable permutations contained in $251364$ is $x+x^2+2x^3+3x^4+3x^5+x^6$, and thus the generating function for the sum completion is $1/(1-x-x^2-2x^3-3x^4-3x^5-x^6)$, which has the growth rate stated in the proposition.
\end{proof}

\subsection{Sums of Wedge Alternations}\label{sec-large-bounded}

In order to use Theorem~\ref{thm-bounded-subst-depth}, we need to enumerate the sum completion of wedge alternations.

\begin{proposition}\label{prop-sum-wedge}
If the permutation class $\C$ contains arbitrarily long sums or skew sums of arbitrarily long wedge alternations then $\lgr(\C)\ge 1+\varphi\approx 2.61803$ (here $\varphi$ denotes the golden ratio).
\end{proposition}
\begin{proof}
By symmetry, let us suppose that $\C$ contains arbitrarily long sums of arbitrarily long wedge alternations.  There are four ways to orient a wedge alternation, $<$, $>$, $\vee$, and $\wedge$, and so we need to count the nonempty sum indecomposable permutations in the closures of each orientation.  A permutation in the closure of a wedge alternation is sum indecomposable if and only if it is skew decomposable, so let $f_\oplus$ denote the generating function for nonempty sum decomposable permutations in the closure of a wedge alternation, $f_\ominus$ the generating function for nonempty skew decomposable permutations, and $f$ the generating function for all nonempty permutations.  Hence we have $f=x+f_\oplus+f_\ominus$ and $f_\ominus=f_\oplus$.

For wedge alternations oriented as $<$, we have that $\pi\ominus 1$ and $\pi\oplus 1$ lie in its closure whenever $\pi$ lies in its closure, and similar rules apply for all other orientations.  It follows that no matter how the wedge alternations are oriented, we have $f_\oplus=f_\ominus=xf=(x+x^2)/(1-2x)$.  Therefore the generating function for the sum completion of some orientation of wedge alternations is $(1-2x)/(1-3x+x^2)$, from which the proposition follows.
\end{proof}

\subsection{Monotone Grid Classes of Vectors, and Their Subclasses}\label{sec-alt-subclasses}

Monotone grid classes of row vectors were first studied by Atkinson, Murphy, and Ru\v{s}kuc~\cite{atkinson:partially-well-:} and Albert, Atkinson, and Ru\v{s}kuc~\cite{albert:regular-closed-:}, who introduced a encoding for such classes (under the name ``$W$-classes'').  Let $\M$ be a $t\times 1$ monotone grid class, meaning that every cell of $\M$ is labeled by $\Av(12)$, $\Av(21)$, or $\emptyset$.  Their encoding associates to each gridded permutation $\pi\in\Grid(\M)$ of length $n$ the word $w_\pi=w_\pi(1)\cdots w_\pi(n)$ where $w_\pi(i)=k$ if the entry $i$ lies in cell $(k,1)$.  The following two propositions are almost certainly folklore, but this author knows them from Albert~[personal communication].


\begin{proposition}\label{prop-subword-closed-form}
Let $A$ be a finite alphabet.  Every subword-closed language $L\subseteq A^*$ can be expressed as a finite union of regular expressions of the form $w^{(1)}A_1^*\cdots w^{(k)}A_k^*w^{(k+1)}$ where each $w^{(i)}$ is a word and each $A_i$ is a subset of $A$.
\end{proposition}
\begin{proof}
First we observe that subword-closed languages are regular.  By Higman's Theorem, there is a finite set of words $B$ such that $w\in L$ if and only if $w\not\ge b$ for any $b\in B$ (this is the word analogue of what we call the basis in the permutation context).  For any $b=b_1\cdots b_k\in B$, the set of words containing $b$ can be expressed as a regular expression $A^*b_1\cdots A^*b_kA^*$, so its complement (the set of words not containing $b$) is also regular, from which it follows that the set of words not containing any $b\in B$ is regular.

Our proof can therefore use induction on the regular expression defining $L$.  The base cases where $L$ is empty or a single letter are both trivial.  Thus we need only consider where $L$ is a union, concatenation, or star of two (or one, in the star case) regular expressions.  The union and concatenation cases are clear, so suppose that $L=E^*$ for a regular expression $E$.  Since $L$ is closed under subwords, it follows that $L=A_E^*$ where $A_E\subseteq A$ is the set of letters occurring in $E$, proving the proposition.
\end{proof}

From this, we immediately have the following growth rate result.

\begin{proposition}\label{prop-subword-closed-growth}
Let $A$ be a finite alphabet.  Every subword-closed language $L\subseteq A^*$ has an integral growth rate.
\end{proposition}
\begin{proof}
As in Proposition~\ref{pwo-atomic-gr}, it follows that the (upper, lower, or proper) growth rate of a $L$ is the maximum (upper, lower, or proper) growth of one of the regular expressions guaranteed by Proposition~\ref{prop-subword-closed-form}; suppose that the growth rate of $L$ is equal to the growth rate of $w^{(1)}A_1^*\cdots w^{(k)}A_k^*w^{(k+1)}$.  Clearly the lower growth rate of $L$ is at least $\max |A_i|$, and an analogous computation to Proposition~\ref{gridded-gr} then shows that the upper growth rate of $L$ is at most $\max |A_i|$, completing the proof.
\end{proof}

Returning to monotone grid classes of vectors, we have, via Proposition~\ref{gridded-gr}, the following result.

\begin{proposition}\label{prop-gr-mono-vector}
If the permutation class $\C$ is contained in a monotone grid class of a vector, then $\gr(\C)$ exists and is integral.
\end{proposition}

Before concluding this section, we note that --- in sharp contrast to grid classes in general --- it is easy to compute bases of grid classes of vectors by iterating the following result.

\begin{proposition}[Atkinson~\cite{atkinson:restricted-perm:}]\label{prop-juxta-basis}
Let $\C$ and $\D$ be permutation classes.  The basis elements of the class $\Grid\left(\begin{array}{cc}\C&\D\end{array}\right)$ can all be expressed as concatenations $\rho\sigma\tau$ where either:
\begin{itemize}
\item[(a)] $|\sigma|=1$, $\rho \sigma$ has the same relative order as a basis element of $\C$, and $\sigma\tau$ has the same relative order as a basis element of $\D$, or
\item[(b)] $\sigma$ is empty, $\rho$ has the same relative order as a basis element of $\C$, and $\tau$ has the same relative order as a basis element of $\D$.
\end{itemize}
In particular, if two classes have finite bases then their juxtaposition also has a finite basis.
\end{proposition}

\subsection{Triple Alternations}\label{subsec-triple-alternations}

We consider first the linear triple alternations, and (without loss of generality) those which can be divided by vertical lines into three parts, left, middle, and right.  Consider such a permutation $\pi_1$, in which each of the parts have $3m^8$ entries.  By the Erd\H{o}s-Szekeres Theorem, the entries of the left part contain a monotone subsequence of length at least $m^4$.  Consider then the linear triple alternation $\pi_2\le\pi_1$ which contains each of these $m^4$ entries together with another $m^4$ entries from each of the other parts.  By applying the Erd\H{o}s-Szekeres Theorem again to the entries in the middle part, we find a monotone subsequence with $m^2$ entries.  Now consider the linear triple alternation $\pi_3\le\pi_2$ which contains these $m^2$ entries, together with $m^2$ entries from each of the other two parts.  By applying the Erd\H{o}s-Szekeres Theorem a third time, to the entries in the right part, find a linear triple alternation with $3m$ entries in which each of the parts is monotone, giving the following result.

\begin{proposition}\label{contain-mono-linear-triples}
If a permutation class contains arbitrarily long linear triple alternations, then it contains a $1\times 3$ or $3\times 1$ monotone grid class.
\end{proposition}

We are only interested in the lower growth rates of these classes.  Proposition~\ref{contain-mono-linear-triples} and our work in the previous subsection give the following.

\begin{proposition}\label{prop-linear-triple-alternations}
If the permutation class $\C$ contains arbitrarily long linear triple alternations then $\lgr(\C)\ge 3$.
\end{proposition}

Now consider the hook triple alternations.  The following proposition follows from the same multiple applications of Erd\H{o}s-Szekeres as Proposition~\ref{contain-mono-linear-triples}.

\begin{proposition}\label{contain-mono-hook-triples}
If a permutation class contains arbitrarily long hook triple alternations, then it contains a symmetry of one of the following classes:
$$
\begin{array}{lll}
\Grid\left(\begin{footnotesize}\begin{array}{cc}\Av(21)&\\\Av(21)&\Av(21)\end{array}\end{footnotesize}\right),
&
\Grid\left(\begin{footnotesize}\begin{array}{cc}\Av(21)&\\\Av(21)&\Av(12)\end{array}\end{footnotesize}\right),
&
\Grid\left(\begin{footnotesize}\begin{array}{cc}\Av(12)&\\\Av(21)&\Av(12)\end{array}\end{footnotesize}\right),
\\\\
\Grid\left(\begin{footnotesize}\begin{array}{cc}\Av(21)&\\\Av(12)&\Av(21)\end{array}\end{footnotesize}\right),
&
\Grid\left(\begin{footnotesize}\begin{array}{cc}\Av(21)&\\\Av(12)&\Av(12)\end{array}\end{footnotesize}\right),
\mbox{ or}
&
\Grid\left(\begin{footnotesize}\begin{array}{cc}\Av(12)&\\\Av(12)&\Av(12)\end{array}\end{footnotesize}\right).
\end{array}
$$
\end{proposition}

Again it suffices to enumerate gridded permutations in these classes, and again, the answer is the same for all of the classes, so we study
$$
\Grid\left(\begin{footnotesize}\begin{array}{cc}&\Av(21)\\\Av(21)&\Av(21)\end{array}\end{footnotesize}\right),
$$
using a special case of an encoding from Vatter and Waton~\cite{vatter:on-points-drawn:}.  

We divide the entries of such a gridded hook triple alternation, say $\pi$, into those which are in the top (designated by $t$), left ($\ell$), and hook ($h$) parts, aiming to find a corresponding word over the alphabet $A=\{h,\ell,t\}$.  We first read the $h$ and $t$ entries from left-to-right, recording $h$s and $t$s, to form a word $w_\pi^{ht}$.  We then read the $h$ and $\ell$ entries from bottom-to-top, recording a word $w_\pi^{h\ell}$.  Because the hook entries are monotone increasing, they correspond to the same entries in each of these two words.  We use this to amalgamate the words, identifying them along their $h$ entries and placing $\ell$s before $t$s.  For example, if $w_\pi^{ht}=t^{i_0}ht^{i_1}\cdots ht^{i_{k}}$ and $w_\pi^{h\ell}=\ell^{j_0}h\ell^{j_1}\cdots h\ell^{j_{k}}$ (here the exponents are allowed to be $0$) then our amalgamated word would be $w_\pi=\ell^{j_0}t^{i_0}h\ell^{j_1}t^{i_1}\cdots h\ell^{i_{k}}t^{i_{k}}$.

This establishes a bijection between gridded permutations (which have the same growth rate as the grid class by Proposition~\ref{gridded-gr}) in this grid class and the language $A^*\setminus A^*t\ell A^*$, from which it follows that the generating function for such gridded permutations is $1/(1-3x+x^2)$.
We thus have the following proposition%
\footnote{By choosing a {\it canonical gridding\/} for each permutation in this grid class it is possible to give the exact enumeration (which was computed by Waton~\cite{waton:on-permutation-:} using another method).  We can do this by placing the line separating $\ell$s from $h$s as far to the right as possible, and then placing the line separating $t$s from $h$s as high as possible.  It follows that in any canonical gridding of a permutation, if there is an $h$, then the $h$s and $\ell$s must contain a descent, meaning that there must be some $h$ before some $\ell$ and therefore we forbid words of the form $\{\ell,t\}^*h\{h,t\}^*$.  Similarly, if there is a $t$, then the $t$s and $h$s must form a descent, so we forbid words of the form $\{h,\ell\}^*t\{l,t\}^*$.  Finally, words of the form $\ell^*tA^*$ correspond to the same permutations as words of the form $\ell^*A\ell$, so they are not canonical.  This shows that the permutations in this grid class are in bijection with the language
$$
A^*\setminus
\left(
A^*t\ell A^*
\cup
\{\ell,t\}^*h\{h,t\}^*
\cup
\{h,\ell\}^*t\{l,t\}^*
\cup
\ell^*tA^*
\right).
$$
This language (and thus the grid class of interest) has the generating function
$$
\frac{1-7x+19x^2-25x^3+17x^4-4x^5}{(1-x)^3(1-2x)(1-3x+x^2)}.
%
$$
Furthermore, it can be computed using Proposition~\ref{prop-juxta-basis} that the basis of this grid class is $\{321, 214365, 214635, 314265, 314625\}$.}.

\begin{proposition}\label{prop-hook-triple-alternations}
If the permutation class $\C$ contains arbitrarily long hook triple alternations then $\lgr(\C)\ge 1+\varphi\approx 2.61803$ (here $\varphi$ denotes the golden ratio).
\end{proposition}

There is a pleasing generalization of Proposition~\ref{prop-hook-triple-alternations}.  We say that a {\it monotone staircase grid class\/} is a grid class of the form
$$
\makeatletter
\def\Ddots{\mathinner{\mkern1mu\raise\p@
\vbox{\kern7\p@\hbox{.}}\mkern2mu
\raise4\p@\hbox{.}\mkern2mu\raise7\p@\hbox{.}\mkern1mu}}
\makeatother
\Grid\left(\begin{footnotesize}
\begin{array}{ccccccc}
&&&&&\Av(21)&\Av(21)\\
&&&&\Av(21)&\Av(21)\\
&&&\Ddots&\Ddots\\
&&\Av(21)&\Av(21)\\
&\Av(21)&\Av(21)\\
\Av(21)&\Av(21)
\end{array}
\end{footnotesize}\right).
$$
Consider a monotone staircase grid class with $k$ nonempty cells.  By Proposition~\ref{gridded-gr}, it suffices to count gridded permutations, and it follows quickly from this that the growth rate of this monotone staircase grid class is given by maximizing
$$
\lim_{n\rightarrow\infty}\sqrt[n]{{a_1n+a_2n\choose a_1n}{a_2n+a_3n\choose a_2n}{a_3n+a_4n\choose a_3n}\cdots{a_{k-1}n+a_k\choose a_{k-1}}}
$$
for $a_1+\cdots+a_k=1$.  (Essentially, this expression counts the number of permutations of length $n$ with $\lfloor a_in\rfloor$ entries in cell $i$, where the cells are labeled from one corner to the other.)  By applying Stirling's formula, taking the $n$th root, and using Lagrange multipliers, this expression is maximized when
$$
\frac{a_1+a_2}{a_1}
=
\frac{(a_1+a_2)(a_2+a_3)}{a_2^2}
=
\frac{(a_2+a_3)(a_3+a_4)}{a_3^2}
=
\cdots
=
\frac{a_{k-1}+a_k}{a_k}.
$$
Setting $r_i=a_{i+1}/a_i$ and $r_1=r$, we find that
$$
r_i=r-1-\frac{1}{r_{i-2}},
$$
for $i\ge 3$, leading to the following result.

\begin{theorem}[Albert and Vatter, unpublished]
The growth rate of a monotone staircase grid class with $k$ cells is
$1+r$ where $r$ is the largest positive root of
$$
x-\cfrac{1}{(x-1)-\cfrac{1}{(x-1)-\cfrac{1}{-\cdots-\cfrac{1}{(x-1)}}}}
$$
if $k$ is even and
$$
x-\cfrac{1}{(x-1)-\cfrac{1}{(x-1)-\cfrac{1}{-\cdots-\cfrac{1}{(x-1)-\cfrac{1}{x}}}}}
$$
if $k$ is odd, where $(x-1)$ occurs $\lfloor k/2\rfloor$ times.
\end{theorem}

\subsection{$(3,1)$-Alternations}\label{subsec-up-down-interleaving-growth}

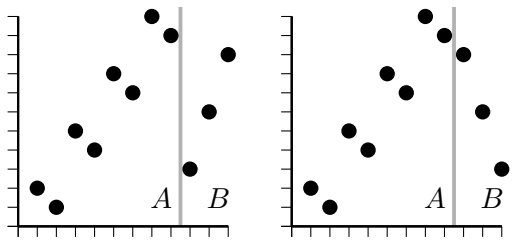
\begin{figure}
\begin{center}
\begin{tabular}{ccc}
\psset{xunit=0.01in, yunit=0.01in}
\psset{linewidth=1\psxunit}
\begin{pspicture}(0,0)(110,110)
\psline[linecolor=darkgray,linestyle=solid,linewidth=0.02in](85,0)(85,115)
\psaxes[dy=10,Dy=1,dx=10,Dx=1,tickstyle=bottom,showorigin=false,labels=none](0,0)(110,110)
\pscircle*(10,20){0.04in}
\pscircle*(20,10){0.04in}
\pscircle*(30,50){0.04in}
\pscircle*(40,40){0.04in}
\pscircle*(50,80){0.04in}
\pscircle*(60,70){0.04in}
\pscircle*(70,110){0.04in}
\pscircle*(80,100){0.04in}
\pscircle*(90,30){0.04in}
\pscircle*(100,60){0.04in}
\pscircle*(110,90){0.04in}
\rput(75,15){$A$}
\rput(105,15){$B$}
\end{pspicture}
&&
\psset{xunit=0.01in, yunit=0.01in}
\psset{linewidth=1\psxunit}
\begin{pspicture}(0,0)(110,110)
\psline[linecolor=darkgray,linestyle=solid,linewidth=0.02in](85,0)(85,115)
\psaxes[dy=10,Dy=1,dx=10,Dx=1,tickstyle=bottom,showorigin=false,labels=none](0,0)(110,110)
\pscircle*(10,20){0.04in}
\pscircle*(20,10){0.04in}
\pscircle*(30,50){0.04in}
\pscircle*(40,40){0.04in}
\pscircle*(50,80){0.04in}
\pscircle*(60,70){0.04in}
\pscircle*(70,110){0.04in}
\pscircle*(80,100){0.04in}
\pscircle*(90,90){0.04in}
\pscircle*(100,60){0.04in}
\pscircle*(110,30){0.04in}
\rput(75,15){$A$}
\rput(105,15){$B$}
\end{pspicture}
\end{tabular}
\end{center}
\caption[]{Two gridded permutations resulting from arbitrarily long $(3,1)$-alternations.}\label{fig-li-ld}
\end{figure}

In this subsection our aim is to give a lower bound on the growth rate of a class which contains arbitrarily long $(3,1)$-alternations.  Recall that such an alternation may be divided by a single horizontal or vertical line into two parts, $A$ and $B$, so that $A$ consists of nonmonotone intervals of length three, each separated from every other by at least one point in part $B$, and every pair of points in part $B$ is separated by at least one of the intervals of part $A$.  Consider a $(3,1)$-alternation, $\pi$, of length at least $4m^4$.  By symmetry we may suppose that it can be separated by a vertical line into two such parts $A$ and $B$.  By the Erd\H os-Szekeres Theorem, there is a monotone subsequence of at least $m^2$ intervals in part $A$; let us suppose by symmetry that this subsequence is increasing.  These nonmonotone intervals are then separated from each other by at least one point in part $B$.  Applying the Erd\H os-Szekeres Theorem to those points we can find a monotone subsequence of length at least $m$.  As each nonmonotone interval in part $A$ contains $21$, we have therefore concluded that every $(3,1)$-alternation of length at least $4m^4$ contains a (symmetry of a) subpermutation of length at least $3m$ which can be divided by a single vertical line into parts $A$ and $B$ in which:
\begin{itemize}
\item part $A$ consists of an increasing set of intervals order isomorphic to $21$, each separated from every other by at least one point in part $B$ and
\item part $B$ consists of a monotone set of points, each separated from every other by at least one interval in $A$,
\end{itemize}
as shown in Figure~\ref{fig-li-ld}.

We therefore need to count permutations contained in the closures of such structures.  Note that these classes are \emph{not} grid classes, because the copies of $21$ on the left are not allowed to interleave with the entries on the right.  Still, it makes sense to refer to the line dividing the $\oplus 21$ side of such permutations from the other side as a ``gridding'', and Proposition~\ref{gridded-gr} easily extends to this context, so we need only count the gridded permutations.  As there are the same number of gridded permutations no matter whether the entries in part $B$ are increasing or decreasing, we assume that they are increasing.

We encode these gridded permutations reading bottom to top, using the alphabet $\{\ell_1, \ell_{21}, r\}$, where $\ell_1$ denotes a single entry on the left, $\ell_{21}$ denotes an interval order isomorphic to $21$ on the left, and $r$ denotes a single entry on the right.  In this manner, the gridded permutations in this class correspond precisely to the language $\{\ell_1,\ell_{21},r\}^*$.  To enumerate these gridded permutations we assign a weight of $x$ to the letters $\ell_1$ and $r$ and a weight of $x^2$ to the letter $\ell_{21}$.  The weight-generating function for this language is then $1/(1-2x-x^2)$, which gives the following proposition%
\footnote{With a bit more work, it is possible to show that this class of permutations is
$$
\Grid\left(\begin{array}{cc}\bigoplus 21&\Av(21)\end{array}\right)\cap\Av(3142, 4213).
$$
Proposition~\ref{prop-juxta-basis} can be used to show that the basis of $\Grid\left(\begin{array}{cc}\bigoplus 21&\Av(21)\end{array}\right)$ is
$$
\{3421, 4132, 4231, 4321, 23154, 24153, 31254, 32154, 34152, 41253, 42153, 43152\}.
$$
The Maple package {\sc InsEnc} described in Vatter~\cite{vatter:finding-regular:} then readily computes that the generating functions for the full grid class and the subclass we are interested in are, respectively,
$$
\frac{1-5x+8x^2-2x^3-2x^4+x^5}{(1-x)^2(1-x-x^2)(1-3x+x^2)}
\quad
\mbox{and}
\quad
\frac{1-3x+x^2+4x^3+x^4}{(1-x)(1-x-x^2)(1-2x-x^2)}.
$$
}.

\begin{proposition}\label{prop-ld-growth}
If the permutation class $\C$ contains arbitrarily long $(3,1)$-alternations then $\lgr(\C)\ge 1+\sqrt{2}$.
\end{proposition}

\subsection{Sum Indecomposable Permutations}\label{subsec-sum-indecomps}

In this final subsection we present lower bounds on the ``diversity'' of sum indecomposable permutations.  These are used in the characterization of sub-$\kappa$ growth rates to verify that certain sequences of sum indecomposable permutations are unachievable.

\begin{proposition}\label{prop-2-sum-indecomps}
For a sum indecomposable permutation $\pi$ of length $n$ either
\begin{enumerate}
\item[(1)] $\pi=n\cdots 21,12\cdots (n-1)\ominus 1,\mbox{ or }1\ominus 12\cdots (n-1)$, or
\item[(2)] $\pi$ contains at least $2$ distinct sum indecomposable permutations of length $n-1$.
\end{enumerate}
In particular, if the permutation class $\C$ contains $2$ sum indecomposable permutations of length $n\ge 4$ then it also contains $2$ sum indecomposable permutations of length $n-1$.
\end{proposition}
\begin{proof}
Take $\pi$ to be a sum indecomposable permutation of length $n$ and denote the indices of the lexicographically minimal path connecting $1$ to $n$ in $G_\pi$ by $1=i_1<i_2<\cdots<i_m=n$.  If this path contains only the vertices $1$ and $n$ then $\pi(1)>\pi(n)$ and it is easy to see that (2) is satisfied if any of the following hold:
\begin{itemize}
\item $\pi$ contains entries both above $\pi(1)$ and below $\pi(n)$,
\item the entries above $\pi(1)$ are nonmonotone,
\item the entries below $\pi(n)$ are nonmonotone, or
\item the entries lying vertically between $\pi(1)$ and $\pi(n)$ are nonmonotone.
\end{itemize}
If none of these conditions hold, it follows either that $\pi$ satisfies (1) or that $\pi=1\ominus 12\cdots(n-2)\ominus 1$, which satisfies (2).

Thus we may suppose that $m\ge 3$.  Now we divide $\pi$ into the sections $\pi((i_k,i_{k+1})\times[n])$.  If all of these sections are empty then $G_\pi$ is a path, so $\pi$ is an increasing oscillation by Proposition~\ref{inc-osc-path}, and it can be checked that (2) holds.  Similarly, (2) is clearly satisfied if two of these sections are nonempty.  Thus we may assume that precisely one of these sections, say $\pi((i_j,i_{j+1})\times[n])$, is nonempty.  In this (final) case, (2) can be seen to hold by the following argument: removing an entry from $\pi((i_j,i_{j+1})\times[n])$ gives one sum indecomposable permutation, and removing either the leftmost, rightmost, top, or bottom entry gives the other.
\end{proof}

Extending the argument of Proposition~\ref{prop-2-sum-indecomps} to handle classes with $3$ sum indecomposable permutations is quite technical, so we take a more computational tack in the proof of the next result.

\begin{proposition}\label{prop-3-sum-indecomps}
If the permutation class $\C$ contains $3$ sum indecomposable permutations of length $n\ge 5$ then it also contains $3$ sum indecomposable permutations of length $n-1$.
\end{proposition}
\begin{proof}
Let $\D$ denote the class of all permutations which each contain at most two sum indecomposable subpermutations of each length.  It is straight-forward to verify that for $n\ge 6$, $\D$ contains at least the following $23$ types of sum indecomposable permutations, which can be partitioned into $6$ families:
\begin{enumerate}
\item[(S1)] $n\cdots 21$, $12\cdots(n-1)\ominus 1$, and $1\ominus 12\cdots(n-1)$, all of which contain a single sum indecomposable subpermutation of each length,
\item[(S2)] all six inflations of $21$ by monotone permutations of lengths $2$ and $n-2$ except $n\cdots 21$,
\item[(S3)] the four inflations $231[1,\alpha,1]$, $231[\alpha,1,1]$, $312[1,\alpha,1]$, and $312[1,1,\alpha]$, where $\alpha=(n-2)\cdots 21$,
\item[(S4)] the four inflations $231[21,\alpha,1]$, $231[\alpha,21,1]$, $312[1,\alpha,21]$, and $312[1,21,\alpha]$, where $\alpha=(n-3)\cdots 21$,
\item[(S5)] the four inflations $2413[\alpha,1,1,1]$, $2413[1,1,1,\alpha]$, $3142[1,\alpha,1,1]$, and $3142[1,1,\alpha,1]$, where $\alpha=12\cdots(n-3)$, and
\item[(S6)] the two increasing oscillations of length $n$.
\end{enumerate}
We claim that all sum indecomposable permutations of length $n\ge 6$ in $\D$ are of one of these types.

As inverse and reverse-complement (the composition of reverse and complement, which commute with each other) preserve sum indecomposability, $\D$ is invariant under these symmetries.  It is readily computed that $\D$ does not contain any of the $24$ permutations $4231$, $24351$, $24513$, $24531$, $25143$, $25314$, $25413$, $35142$, $35412$, $35421$, $45312$, $235164$, $236145$, $236541$, $251364$, $324651$, $456123$, $456321$, $621543$, $2317465$, $2317546$, $3142765$, $3412765$, $3421756$, and therefore $\D\subseteq\Av(B)$ where $B$ denotes the $71$ permutations that can be generated from the aforementioned list of permutations by inverse and reverse-complement.  The Maple package {\sc InsEnc} described in Vatter~\cite{vatter:finding-regular:} can compute that the generating function for the sum indecomposable elements of $\Av(B)$ is
$$
1+x+x^2+3x^3+12x^4+25x^5+23\frac{x^6}{1-x},
$$
thereby verifying that the sum indecomposable elements of $\D$ are of the form claimed.

Now suppose that the permutation class $\C$ contains $3$ sum indecomposable permutations of length $n\ge 6$ (the $n=5$ case can be checked by hand or computer).  If $\C$ contains a sum indecomposable permutation which does not lie in $\D$ then this sum indecomposable permutation (and thus $\C$) contains at least $3$ sum indecomposable permutations of length $n-1$, so we may assume that the sum indecomposable permutations of $\C$ are contained in $\D$.  Now note that each sum indecomposable permutation in $\D$ contains a sum indecomposable permutation of length $n-1$ of its same type.  Therefore, no matter which $3$ sum indecomposable permutations of $\D$ are contained in $\C$, $\C$ must also contain at least $3$ sum indecomposable permutations of length $n-1$.
\end{proof}

\bigskip
\noindent{\bf Acknowledgments:} The author owes a large debt to Michael Albert for numerous insightful discussions related to this work.  Thanks are also owed to Robert Brignall, Nik Ru\v{s}kuc, and the referee for their helpful comments and suggestions, and Steven Finch for spotting an error in an earlier version of the paper.  Finally, the author owes his knowledge of the work of Gy{\'a}rf{\'a}s and Lehel~\cite{gyarfas:a-helly-type-pr:} and K{\'a}rolyi and Tardos~\cite{karolyi:on-point-covers:} to an unpublished manuscript of Daniel Werner and Matthias Lenz.

\bibliographystyle{acm}
\bibliography{../../refs}

\end{document}